\documentclass[12pt]{article}

\usepackage{amsthm,amsmath,amssymb,enumerate}
\usepackage[colorlinks=true,linkcolor=blue, citecolor=blue,
urlcolor=blue, linktocpage=true]{hyperref}
\usepackage{fullpage}
\usepackage{natbib}

\theoremstyle{plain}  
\newtheorem{theorem}{Theorem}[section] 
\newtheorem{lemma}[theorem]{Lemma} 
\newtheorem{proposition}[theorem]{Proposition}

\theoremstyle{definition} 
\newtheorem{definition}[theorem]{Definition}
\newtheorem{example}[theorem]{Example}
\newtheorem{remark}[theorem]{Remark}

\numberwithin{equation}{section}

\newcommand{\E}{\mathbf{E}}
\renewcommand{\P}{\mathbf{P}}
\newcommand{\bX}{{\boldsymbol{X}}}
\newcommand{\bY}{{\boldsymbol{Y}}}
\newcommand{\bZ}{{\boldsymbol{Z}}}
\newcommand{\bC}{{\boldsymbol{C}}}
\newcommand{\bB}{{\boldsymbol{B}}}
\renewcommand{\bB}{{\mathbb{B}}}
\newcommand{\BX}{\bB_\bX}
\newcommand{\BY}{\bB_\bY}
\newcommand{\Prob}[1]{\P\left\{#1\right\}}
\newcommand{\one}{{\boldsymbol{1}}}

\newcommand{\sB}{\mathcal{B}}

\newcommand{\sF}{\mathcal{F}}

\newcommand{\R}{\mathbb{R}} 
\newcommand{\Rb}{\bar{\R}}

\newcommand{\Sphere}[1][d-1]{\mathbb{S}^{#1}}
\newcommand{\cone}{\mathbb{C}}

\newcommand{\ve}{\mathsf{e}}
\newcommand{\ue}{\mathsf{u}}
\newcommand{\gae}{\mathsf{g}}
\newcommand{\Tg}{\mathsf{T_g}}
\newcommand{\Gae}{\mathsf{G}}
\newcommand{\quantile}[1][\alpha]{\mathsf{q}_{#1}}

\newcommand{\salg}{\mathfrak{F}}
\newcommand{\ssalg}{\mathfrak{A}}
\newcommand{\Lp}[1][p]{\mathsf{L}^{#1}}
\newcommand{\Lpb}[1][p]{\overline{\mathsf{L}}^{#1}}
\newcommand{\Mp}[1][p]{\mathsf{M}^{#1}}
\newcommand{\core}{\mathsf{m}}
\newcommand{\chull}{\mathsf{M}}

\DeclareMathOperator{\clo}{cl}
\DeclareMathOperator{\esssup}{ess\,sup}
\DeclareMathOperator{\essinf}{ess\,inf}

\usepackage{fancybox}
\newlength{\querylen}
\usepackage[colorinlistoftodos,textsize=tiny]{todonotes}
\newcommand{\Comments}{1}
\newcommand{\mynote}[2]{\ifnum\Comments=1\textcolor{#1}{#2}\fi}
\newcommand{\mytodo}[2]{\ifnum\Comments=1%
  \todo[linecolor=#1!80!black,backgroundcolor=#1,bordercolor=#1!80!black]{#2}\fi}

\ifnum\Comments=1               
\fi

\ifnum\Comments=1               
\fi

\ifnum\Comments=1               
\fi

\renewcommand{\subset}{\subseteq}

\begin{document}

\title{Set-valued conditional functionals of random sets}

\author{Tobias Fissler\footnotemark[1] \and Ilya
  Molchanov\footnotemark[2]}
\footnotetext[1]{tobias.fissler@math.ethz.ch, RiskLab, ETH Zurich, 8092 Zurich, Switzerland}
\footnotetext[2]{ilya.molchanov@unibe.ch, Institute of Mathematical
  Statistics and Actuarial Science, University of Bern, 3012 Bern,
  Switzerland}

\date{\today}
\maketitle

\begin{abstract}
  Many key quantities in statistics and probability theory such as the
  expectation, quantiles, expectiles and many risk measures are
  law-determined maps from a space of random variables to the reals.
  We call such a law-determined map, which is normalised, positively
  homogeneous, monotone and translation equivariant, a gauge function.
  Considered as a functional on the space of distributions, we can
  apply such a gauge to the conditional distribution of a random
  variable. This results in conditional gauges, such as conditional
  quantiles or conditional expectations.  In this paper, we apply such
  scalar gauges to the support function of a random closed convex set
  $\bX$. This leads to a set-valued extension of a gauge function.  We
  also introduce a conditional variant whose values are themselves
  random closed convex sets. In special cases, this functional becomes
  the conditional set-valued quantile or the conditional set-valued
  expectation of a random set.  In particular, in the unconditional
  setup, if $\bX$ is a random translation of a deterministic cone and
  the gauge is either a quantile or an expectile, we recover the cone
  distribution functions studied by Andreas Hamel and his co-authors.
  In the conditional setup, the conditional quantile of a random
  singleton yields the conditional version of the half-space
  depth-trimmed regions.\\
  

  
\end{abstract}

\section{Introduction}
\label{sec:introduction}

Due to the lack of a natural total order, many classical concepts for
random variables are not directly applicable for random vectors. In
the course of extending such concepts to random vectors the classical definitions often
acquire new features. The most well-known example is that of the
median and of quantiles: while on the line they are numbers, in the space
they are sometimes defined as sets or collections of numbers depending
on direction, see \cite{hallin10} and
\cite{paindaveine21}. Alternatively, it is possible to work with
partial orders as done by \cite{belloni11}, or to apply tools from the
theory of optimal transport as pursued by \cite{carlier16:_vector}, or
to derive a multivariate median from a solution of a certain minimisation
procedure, see \cite{chaudhuri96}.

A useful replacement of quantiles in multivariate statistics is
provided by \emph{depth-trimmed regions}. In this course, the
distribution of a random vector or its empirical counterpart is
associated with a nested family of sets, depending on the parameter
$\alpha \in (0,1]$. These depth-trimmed regions are exactly the upper
level sets of the depth function, which allocates to each point its
depth with respect to the underlying probability distribution. The
first such concept is the half-space depth introduced by \cite{tuk75}
(see \cite{nag:sch:wer19} for a recent survey),
later followed by an axiomatic approach of \cite{zuo00:_gener} and
numerous further concepts including the simplicial depth due to
\cite{liu88} and zonoid depth defined by \cite{kos:mos97} and \cite{mos02}.
\cite{mol:tur21} showed that many such constructions arise from
an application of a sublinear expectation to the projections of random
vectors and interpreting the obtained function as the support function
of the depth-trimmed region.

Motivated by financial applications, it is natural to associate with a
vector a certain random closed convex set which describes all
financial positions attainable from the value of the vector by
admissible transactions. If all transactions are forbidden, this
random set is $X+\R_-^d$. The most classical nontrivial example arises
when the components of a vector $X$ represent different currencies,
and the set of admissible positions is obtained by adding to the
vector all points from a cone $\cone$ which describes positions
attainable from price zero, see \cite{kab:saf09}. In the financial
setting the cone $\cone$ always contains $\R_-^d$, meaning that
disposal or consumption of assets are always allowed.

A number of works has addressed assessing the risk of $X+\cone$, which
is the cone translated by $X$, see \cite{ham:hey10} and
\cite{mol:cas16}. The key issue is to associate with $X+\cone$ a
closed convex set in $\R^d$ which describes its risk and regards
$X+\cone$ as acceptable if this set contains the origin. This set can
be described as the family of all $a\in\R^d$ such that $a+X+\cone$ is
acceptable.

Abstracting away from the financial interpretation, it is possible to
associate with each point $x$ from $\R^d$ its position in relation to
$X+\cone$. If $\cone=\{0\}$ is degenerate, then this is similar to
considering the depth of $x$ in relation to the distribution of the
random vector $X$. In case of a nontrivial cone $\cone$, its
orientation can be rather arbitrary and can address other features, not
necessarily inherent to the financial setting. This line of research
was initiated by \cite{ham:kos18} who defined the \emph{cone
  distribution} function $F_{X,\cone}(z)$, $z\in\R^d$, as the infimum of the
probability content of closed half-spaces which pass through $z$ and
have outer normals from the polar cone to $\cone$. The excursion sets
of the cone distribution function are termed the
\emph{$\cone$-quantile function} of $X$. There is a one-to-one
correspondence between such set-valued quantile functions and the cone
distribution function.

A similar programme has been carried over by \cite{linh:hamel23} for
\emph{expectiles} replacing quantiles. The sublinearity
property of expectiles established by \cite{BelliniKlarETAL2014} makes
it possible to naturally preserve the convexity property instead of
taking the convex hull involved in the constructions based on quantiles.

In this paper we extend the mentioned works of Andreas Hamel and his
coauthors in three directions.  First, instead of randomly translated
cones, we allow for \emph{random closed convex sets} $\bX$ in $\R^d$,
not necessarily being a random translation of a deterministic cone.
Second, instead of considering a set-valued version of quantiles or
expectiles, we construct a set-valued version of a general \emph{gauge
  function} $\gae$. A general gauge function is a law-determined
real-valued mapping from the space of random variables, which is
constant preserving, positively homogeneous, monotone and translation
equivariant. Besides the prime examples of quantiles and expectiles,
our definition of the gauge also comprises the average-value-at-risk
or expected shortfall, an important risk measure in quantitative risk
management, or the gauges based on moments as defined by \cite{fis03}.
Third, we construct a \emph{conditional version} of the set-valued
gauge of the random set $\bX$. That means, for a sub-$\sigma$-algebra
$\ssalg$ of the $\sigma$-algebra of the underlying probability space,
a gauge function $\gae$, and a random closed convex set $\bX$, we
construct another random $\ssalg$-measurable closed convex set
$\bY=\Gae(\bX|\ssalg)$.  In the unconditional case and if the random
set is a deterministic cone translated by a random vector, our
set-valued quantiles and expectiles yield the objects defined by
\cite{ham:kos18} and \cite{linh:hamel23}.

We first explain our construction without conditioning. 
Recall that a closed convex set $F\subseteq\R^d$ can be uniquely
represented as the intersection of all half-spaces containing $F$
defined via supporting hyperplanes of $F$ with any normal vector
$w$ from the unit sphere $\Sphere$. That is,
\[
  F = \bigcap_{w\in \Sphere} H_w \big (h(F,w) \big),
\]
where 
\begin{displaymath}
  h(F,w)=\sup \big\{\langle w,x\rangle: x\in F \big\},\quad w\in\R^d,
\end{displaymath}
is the support function of $F$ (also called a scalarisation) and a
generic half-space with normal vector $w\in\R^d$ is defined by
\begin{equation}
  \label{eq:Hwt}
  H_w(t)=\big\{x\in\R^d:\langle x,w\rangle\leq t\big\}, \qquad t\in\R.
\end{equation}
Note that $H_w(\infty)=H_0(t)=\R^d$ for any $w\in\R^d$ and $t\in\R$.

This approach is also applicable if a deterministic $F$ is replaced by
a compact convex random set $\bX$. In this case, the support function
$h(\bX,w)$ is a path continuous random function of $w$ from 
the unit sphere.  The values $h(\bX,w)$, $w\in\Sphere$, make it
possible to describe $\bX$ in terms of a collection of random
variables, sometimes called scalarisations of $\bX$.  We can thus
apply the gauge function $\gae$ (e.g., a quantile)
to each of these scalarisations,
considering $\gae(h(\bX,w))$, $w\in\Sphere$.  Hence, we end up with
the \emph{unconditional} set-valued gauge, which is the closed random
set
\begin{equation}
  \label{eq:gauge1}
  \Gae(\bX) = \bigcap_{w\in \Sphere} H_w \big (\gae(h(\bX,w)) \big).
\end{equation}

In order to build the corresponding conditional version, we first need
to define the conditional scalar gauge function with respect to a
sub-$\sigma$-algebra $\ssalg$.  Since the gauge itself is
law-determined, this can be achieved by applying the gauge to the
conditional distribution.  Hence, we replace the gauge of the
scalarisation, $\gae(h(\bX,w))$, $w\in\Sphere$, in \eqref{eq:gauge1}
by the conditional gauge of the scalarisation.  This leads us to the
random closed convex set
\[
  \bigcap_{w\in \Sphere} H_w \big (\gae(h(\bX,w)|\ssalg) \big),
\]
While we deal with uncountable intersections, it is known that
the largest $\ssalg$-measurable random closed convex subset of this
intersection is well defined \citep{lep:mol19}.

This construction relies on letting the direction $w$ be
deterministic. This may lead to problems when $\bX$ is unbounded and
so the support function in some (possibly, random) directions becomes
infinite.  For instance, if $\bX$ is a half-space with the normal
having a non-atomic distribution, then $h(\bX,w)=\infty$ almost surely
for each $w\in\Sphere$, so that $\bX$ cannot be recovered by the
values of the support function on any countable subset of $\Sphere$.
We can address this problem by replacing the deterministic
$w\in\Sphere$ by all $\ssalg$-measurable random vectors. With this
idea, our set-valued \emph{conditional} gauge functions is the largest
$\ssalg$-measurable random closed convex set contained in the
intersection of $H_W(\gae(h(\bX,W)|\ssalg))$ over all
$\ssalg$-measurable $W$ with values in $\Sphere$.

As the outcome of our construction, a random closed convex set is associated
with a (possibly, random in the conditional setting) set which serves
as a kind of a depth-trimmed region for the random set. This
construction is principally different from the one developed by
\cite{cascos21:_depth}, where a depth-trimmed region for a random set
is a collection of sets. 

Following the works by 
\cite{ham:hey10}, \cite{ham:hey:rud11}, \cite{ham:rud:yan13}, and \cite{mol:cas16},
in the special cases of subadditive or superadditive gauge functions
and in the unconditional setting, \cite{mol:mueh21}
systematically studied subadditive and superadditive
functionals of random closed convex sets. The
conditional setting calls for new tools which are developed in the
current paper. Working out
these conditional versions makes it possible to calculate conditional
risks given some information about the financial position or about the
market.  The importance of such a conditional approach in risk
measurement has been emphasised in \cite[Chapter~9.2]{MFE15}.

The paper is organised as follows. Section~\ref{sec:cond-gauge-funct}
introduces gauge functions, their conditional variants and provides
several examples of important gauge
functions. Section~\ref{sec:random-convex-sets} provides a necessary
background material on random closed convex sets and several concepts
related to them. A representation theorem using the intersection of a
countable number of random half-spaces is also
proved. Section~\ref{sec:cond-set-valu} describes the construction of
a set-valued conditional gauge. Particular cases resulting from
choosing some of the most important gauge functions are discussed in
Section~\ref{sec:special-cases-gauges}. In particular, choosing the
essential infimum or essential supremum as the gauge leads to concepts
close to the ideas of the conditional core and conditional convex hull
developed by \cite{lep:mol19}. Section~\ref{sec:random-singletons}
discusses conditional gauges for random singletons and thus provides
conditional versions of some well known depth-trimmed
regions. For instance, we obtain conditional quantiles, 
mentioned by \cite{hallin10}, \cite{MR4297710} and \cite{MR4297710}.
Section~\ref{sec:cones-transl-rand} concerns 
conditional gauges for (possibly, translated) random cones.

\section{Conditional gauge functions}
\label{sec:cond-gauge-funct}

\subsection{Gauge functions}
\label{sec:gauge-functions}

Let $(\Omega,\salg,\P)$ be a complete atomless probability space, and
let us denote by $\Lp(\R^d)$ the set of all $p$-integrable random
vectors in $\R^d$ (actually, their a.s.\ equivalence classes), where
$p\in[1,\infty]$.  Endow $\Lp(\R^d)$, $p\in[1,\infty]$, with the
$\sigma(\Lp,\Lp[q])$-topology based on the pairing of $\Lp(\R^d)$ and
$\Lp[q](\R^d)$ with $p^{-1}+q^{-1}=1$.  Further, let $\Lp[0](\R^d)$ be
the family of all random vectors with the topology induced by the
convergence in probability, equivalently, $\E\min(|X_n-X|,1)\to0$.
These families become families of random variables if $d=1$.  For a
sub-$\sigma$-algebra $\ssalg$ of $\salg$ we write $\Lp(\R^d;\ssalg)$
to indicate the set of all $\ssalg$-measurable random vectors in
$\Lp(\R^d)$ where $p\in\{0\}\cup[1,\infty]$.  Further, denote by
$\Lpb(\R^d)$ the family of random vectors $\xi$ which are
representable as the sum $\xi=\xi'+\xi''$, where $\xi'\in\Lp(\R^d)$
for $p\in\{0\}\cup[1,\infty]$ and $\xi''$ belongs to the family
$\Lp[0]([0,\infty]^d)$ of random vectors whose components take values
from the extended positive half-line.  Similarly, for a
sub-$\sigma$-algebra $\ssalg$ of $\salg$, the family $\Lpb(\R^d; \ssalg)$
consists of all $\ssalg$-measurable elements of $\Lpb(\R^d)$.

We introduce the following concept of a \emph{gauge
  function}. Denote by $\Rb = [-\infty,\infty]$ the extended real
line. 


\begin{definition}
  \label{def:subl}
  A \textit{gauge function} is a map $\gae:\Lpb(\R)\to\Rb$, satisfying
  the following properties for all $X,Y\in \Lpb(\R)$:
  \begin{enumerate}[(g1)]
  \item law-determined: $\gae(X)=\gae(Y)$ whenever $X$
    and $Y$ share the same distribution;
  \item constant preserving: $\gae(c)=c$ for all
    $c\in\R$;
  \item positive homogeneity: $\gae(cX)=c\,\gae(X)$ for all $c\in(0,\infty)$;
  \item monotonicity: $\gae(X)\leq\gae(Y)$ if $X\leq Y$ almost
      surely;
  \item translation equivariance: $\gae(X+c)=\gae(X)+c$ for all $c\in\R $.    
    \\[2ex]
    In some cases, we also impose the following optional properties:
  \item Lipschitz property (if $p\in[1,\infty]$):
    $|\gae(X)-\gae(Y)|\leq \|X-Y\|_p$;
  \item subadditivity: $\gae(X+Y)\leq \gae(X)+\gae(Y)$ if the
      right-hand side is well defined;
  \item superadditivity: $\gae(X+Y)\geq \gae(X)+\gae(Y)$ if the
      right-hand side is well defined;
  \item sensitivity with respect to infinity:
    $\P\{X=\infty\}>0$ implies that $\gae(X)=\infty$.
  \end{enumerate}
\end{definition}

Note that in (g7) and (g8) the right-hand sides are well defined
unless the expressions $(\infty-\infty)$ or $(-\infty+\infty)$ appear.

\begin{remark}
  It suffices to replace (g2) by the requirement that $\gae(0)$ is
  finite. Indeed, positive homogeneity (g3) implies that
  $\gae(0)=\gae(c\cdot 0) = c\,\gae(0)$ for all
  $c\in(0,\infty)$. Therefore, $\gae(0)\in\{0,\pm\infty\}$.
  Assuming that $\gae(0)$ is finite implies that $\gae(0)=0$, and then
  the translation equivariance yields that $\gae(c)=c$ for all
  $c\in\R$. On the other hand, the translation equivariance implies
  for all $c\in\R$ that
  $\gae(\infty) = \gae(\infty + c) = \gae(\infty)+c$. Hence,
  $\gae(\infty) = \infty$.  Thus, (g3) and (g5) in combination with
  the finiteness of $\gae(0)$ readily imply (g2).
\end{remark}

Since the gauge is law-determined, it is actually a functional on
distributions of random variables. In view of this, the monotonicity
property (g4) holds if $X$ is stochastically smaller than $Y$, meaning
that the cumulative distribution functions satisfy $F_X(t)\ge F_Y(t)$
for all $t\in\R$.  Recall that $X\le Y$ almost surely implies the first
order stochastic dominance.

To ensure better analytical properties, it may be useful to require
that gauge functions are lower semicontinuous in $\sigma(\Lp,\Lp[q])$, that is,
\begin{equation*}
  \gae(X)\leq \liminf_{n\to\infty} \gae(X_n)
\end{equation*}
for each sequence $\{X_n,n\geq1\}\subset \Lpb(\R)$ converging to
$X\in \Lpb(\R)$ in the $\sigma(\Lp,\Lp[q])$-topology if
$p\in[1,\infty]$, and converging in probability if $p=0$. While such
properties are very natural for subadditive gauges, checking and
interpreting them becomes more complicated if subadditivity is not
imposed.

It is easy to see that the translation equivariance property implies
the Lipschitz property of gauge functions defined on
$\Lp[\infty](\R)$.  If the gauge is subadditive, it will be called a
\emph{sublinear expectation} and denoted by $\ve$ (possibly, with some
sub- and superscripts). A monotone, translation equivariant, and
superadditive gauge is termed a \emph{utility function} or a
\emph{superlinear} gauge and is denoted by $\ue$.

If $\gae$ is a gauge function, its dual is defined by
\begin{equation}
  \label{eq:1}
  \overline{\gae}(X)=-\gae(-X), \quad X\in\Lp(\R).
\end{equation} 
Passing to the dual relates sublinear and superlinear gauges.

A superadditive gauge satisfies all properties of a coherent utility
functions, see \cite{Delbaen2012}, where the input $X$ stands for a
gain.  Its subadditive dual is a coherent risk measure in the sign
convention of \cite{MFE15}, where the input $X$ stands for a loss.

\subsection{Concatenation approach}
\label{sec:conc-appr}

For $p\in\{0\}\cup[1,\infty]$, denote by
$\Mp$ the family of Borel probability measures on $\Rb$,
which correspond to distributions of random variables in $\Lpb(\R)$.
The fact that the gauge function is law-determined means that $\gae$
induces a functional $\Tg\colon \Mp\to \Rb$. We equip $\Mp$ with the
smallest $\sigma$-algebra $\sB(\Mp)$ such that for any Borel set
$B\subseteq \Rb$ the evaluation map
$I_B\colon \Mp\to [0,1]$ given by $I_B(\mu) = \mu(B)$ is
Borel measurable.

Let $\ssalg\subseteq \salg$ be a sub-$\sigma$-algebra of $\salg$.  Our
aim is to define a \emph{conditional} gauge function with values in
$\Lp[0](\bar \R;\ssalg)$, which inherits the corresponding properties of
$\gae$ and which coincides with $\gae$ if $\ssalg$ is trivial.  Denote
by $\P_{X|\ssalg}$ a regular version of the conditional distribution
of $X\in\Lpb(\R)$ given $\ssalg$.  Since for $p\in[1,\infty]$ and
$X\in\Lpb(\R)$, the negative part of $X$ is $p$-integrable, we have
that $\P_{X|\ssalg}\in \Mp$ almost surely.  For $p=0$, it trivially
holds that $\P_{X|\ssalg}\in \Mp[0]$ almost surely.

The conditional gauge is defined following the construction described
by \cite{FisslerHolzmann2022} via the concatenation
\begin{math}
  \gae(X|\ssalg) := \Tg(\P_{X|\ssalg}).
\end{math}
To make this approach mathematically meaningful, we need to assume
that $\Tg\colon \Mp\to \Rb$ is $\sB(\Mp)/\sB(\Rb)$-measurable.  This
assumption holds for all practically relevant examples discussed in
Section~\ref{subsec:examples-gauge} due to the results of
\cite{FisslerHolzmann2022}.  This measurability assumption directly
yields the following result, establishing $\ssalg$-measurability
of $\gae(X|\ssalg)$. It follows from Lemma~2 of
\cite{FisslerHolzmann2022}.

\begin{lemma}
  \label{lemma:measurability}
  If $\Tg\colon \Mp\to \Rb$ is $\sB(\Mp)/\sB(\Rb)$-measurable, then
  $\Tg(\P_{X|\ssalg})$ is $\ssalg/\sB(\Rb)$-measurable.
\end{lemma}

The construction of the conditional gauge can be summarised in the
following proposition.

\begin{proposition}\label{prop:conditional-g}
  Let $\gae:\Lpb(\R)\to\Rb$, $p\in\{0\}\cup[1,\infty]$, be a gauge
  function such that $\Tg\colon \Mp\to \Rb$ is
  $\sB(\Mp)/\sB(\Rb)$-measurable. Then, for any sub-$\sigma$-algebra
  $\ssalg\subseteq \salg$, the \emph{conditional gauge function}
  \begin{displaymath}
    \gae(\cdot|\ssalg) : \Lpb(\R)\to \Lp[0](\bar \R;\ssalg),
    \quad X\mapsto \gae(X|\ssalg) := \Tg(\P_{X|\ssalg})
  \end{displaymath}
  is such that for each $X\in \Lpb(\R)$, the conditional gauge
  $\gae(X|\ssalg)$ exists, and it is an
  $\ssalg$-measurable random variable with values in $\Rb$.
\end{proposition}



The properties of Definition~\ref{def:subl} and the construction via
concatenation from Proposition~\ref{prop:conditional-g} directly imply
the following properties.

\begin{proposition}\label{prop:properties}
Under the assumptions of Proposition~\ref{prop:conditional-g},
  the conditional gauge function satisfies the following properties
  for all $X,Y\in \Lpb(\R)$:
  \begin{enumerate}[(i)]
  \item constant preserving: $\gae(X|\ssalg)=X$ a.s.\ if $X$ is $\ssalg$-measurable;
  \item positive homogeneity:
    $\gae(\gamma X|\ssalg)=\gamma\gae(X|\ssalg)$ a.s.\ for all 
    $\ssalg$-measurable $\gamma$ with $\gamma\in\Lp[\infty](\R_+)$ if
    $p\in[1,\infty]$ and $\gamma\in\Lp[0](\R_+)$ if $p=0$;
  \item law-determined: $\gae(X|\ssalg)=\gae(Y|\ssalg)$ a.s.\
    if $\P_{X|\ssalg} = \P_{Y|\ssalg}$ almost
    surely;
  \item monotonicity: $\gae(X|\ssalg)\leq\gae(Y|\ssalg)$ a.s.\ if
    $X\le Y$ in the first order stochastic dominance conditionally on
    $\ssalg$ (that is, almost surely
    $F_{X|\ssalg}(t)\ge F_{Y|\ssalg}(t)$ for all $t\in\R$ for the conditional c.d.f.s of $X$ and $Y$ given $\ssalg$), e.g., if
    $X\leq Y$ a.s.;
  \item conditional translation equivariance:
    \begin{displaymath}
      \gae(X+Y|\ssalg)=\gae(X|\ssalg)+Y\quad a.s.
    \end{displaymath}
    for all $\ssalg$-measurable $Y$ if $\gae(X|\ssalg)>-\infty$.
  \end{enumerate}
\end{proposition}

\begin{remark}\label{remark:extension}
  It is possible to extent the domain of the conditional gauge from
  $\Lpb(\R)$ to all $X\in \Lpb[0](\R)$ such that
  $\P_{X|\ssalg}\in \Mp$ almost surely.  The simplest situation arises
  when $X$ itself is $\ssalg$-measurable so that the conditional
  distribution $\P_{X|\ssalg}$ becomes degenerate and thus an element
  of $\Mp$ for $p\in\{0\}\cup[1,\infty]$.  Hence, property (i) in
  Proposition~\ref{prop:properties} can be extended to
  $\gae(X|\ssalg)=X$ a.s.\ for all $X\in\Lpb[0](\R;\ssalg)$.
  Similarly, the positive homogeneity in (ii) holds for all
  $\gamma\in \Lp[0]([0,\infty);\ssalg)$. Indeed, the conditional
  c.d.f. of $\gamma X$ given $\ssalg$  is
  \begin{displaymath}
    F_{\gamma X|\ssalg}(t;\omega) =
    \begin{cases}
      F_{X|\ssalg}(t/\gamma(\omega); \omega), & \text{if } \gamma(\omega)>0,\\
      1_{[0,\infty)}(t), & \text{if } \gamma(\omega)=0,
    \end{cases}
  \end{displaymath}
  for $\P$-almost all $\omega\in\Omega$.  Moreover, the translation
  equivariance in (v) can be extended to any $\ssalg$-measurable $Y$.
  This argument appears also later in the construction of the
  generalised conditional expectation.
\end{remark}

Further properties like sub- and superadditivity easily translate to
conditional gauges.  For the sensitivity with respect to infinity, (g9), we
obtain the following result.

\begin{lemma}
  Suppose the gauge function satisfies (g9) and the assumptions of
  Proposition~\ref{prop:conditional-g} are satisfied. Then, for
  $X\in \Lpb(\R)$, on the event $\big\{\P\{X=\infty|\ssalg\}>0\big\}$,
  the conditional gauge $\gae(X|\ssalg)$ is $\infty$ almost surely.
\end{lemma}

We close this subsection by a direct observation, which follows from
the fact that $\P_{X|\ssalg} = \P_{X}$ almost surely if $X$ is
independent of $\ssalg$.

\begin{proposition}\label{prop:independence}
  Under the assumptions of Proposition~\ref{prop:conditional-g}, it
  holds that $\gae(X|\ssalg)=\gae(X)$ almost surely for all
  $X\in\Lpb(\R)$ which are independent of $\ssalg$.  In particular,
  $\gae(X|\ssalg)=\gae(X)$ almost surely for all $X\in\Lpb(\R)$ if
  $\ssalg$ is trivial.
\end{proposition}

Following the approach of \cite{MR3734154}, it is alternatively
possible to consider a gauge function as a norm on the module of
$\ssalg$-measurable random variables. However, then a conditional
gauge is not necessarily law-determined. Furthermore, we do not only
work with sublinear gauges, as considered by \cite{MR3734154}.

\subsection{Examples}
\label{subsec:examples-gauge}

Below we mention main examples of gauges and their conditional
variants.

\paragraph{Quantiles.} These are gauges defined on $\Lpb[0](\R)$.

\begin{definition}
  For $\alpha\in(0,1]$, the \emph{lower quantile} of a random variable
  $X\in \Lpb[0](\R)$ is defined as
  \begin{displaymath}
    \quantile^-(X):=\inf\{t\in\R: \Prob{X\leq t}\geq \alpha\}, 
  \end{displaymath}
  and the \emph{upper quantile} is defined for $\alpha\in[0,1)$ as
  \begin{displaymath}
    \quantile^+(X):=\inf\{t\in\R: \Prob{X\leq t}> \alpha\}.
  \end{displaymath}
\end{definition}

As usual, we set $\inf \emptyset:=\infty$ and $\inf \R := -\infty$.
As shown by \cite{FisslerHolzmann2022}, quantile functionals are
$\sB(\Mp)/\sB(\Rb)$-measurable. They satisfy properties
(g1)\,--\,(g5), but not (g9). For $\alpha\in(0,1)$ they are neither
sub- nor superadditive. We refer to \cite{Castro:2023} for an
exhaustive overview of the properties of conditional quantiles. Even
though conditional quantiles are introduced differently, Theorem~2.6
from the cited paper shows that they coincide with those obtained
by our concatenation approach.

A special case of the lower quantile with $\alpha=1$ yields the
\emph{essential supremum} 
\begin{displaymath}
  \esssup X:= \quantile[1]^-(X),
\end{displaymath}
which is subadditive. Letting $\alpha=0$ in the upper quantile yields
the \emph{essential infimum}
\begin{displaymath}
  \essinf X:= \quantile[0]^+(X),
\end{displaymath}
which is superadditive.
These two functions are dual to each other. 
Moreover, it is easy to show that each
gauge function satisfies
\begin{displaymath}
  \essinf X\leq \gae(X)\leq \esssup X.
\end{displaymath}
Alternatively to the concatenation approach,
a construction of the conditional essential supremum (or infimum) can
be found in \citet[Appendix~A.5]{FoellmerSchied2004}.

\paragraph{Generalised expectation.}

Another typical choice for a gauge is the expectation, which is well
defined on $\Lpb[1](\R)$, it satisfies all properties (g1)\,--\,(g9)
and the expectation functional is $\sB(\Mp)/\sB(\Rb)$-measurable due
to \cite{FisslerHolzmann2022}.  The corresponding conditional gauge
function is the well known conditional expectation.  As detailed in
Remark~\ref{remark:extension}, we can extend the domain of the
conditional expectation to all $X\in\Lpb[0](\R)$ such that
$\P_{X| \ssalg}\in\Mp[1]$ almost surely.  This basically amounts to
considering the \emph{generalised conditional expectation}, see, e.g.,
\citet[Appendix~B]{lep:mol19}.


\paragraph{Average quantiles.}
For a fixed value of $\alpha\in[0,1)$ and $X\in\Lpb[0](\R)$, define
the \emph{right-average quantile} 
\begin{equation}
  \label{eq:3}
  \ve_\alpha(X):=\frac{1}{1-\alpha}
  \int_{\alpha}^{1}\quantile[t]^-(X) dt.  
\end{equation}
Since $\quantile[t]^-(X)< \quantile[t]^+(X)$ only for at most countably many
$t\in(0,1)$ it is immaterial if the integrand is the lower or upper
quantile function.  
If $X$ describes a loss, then $\ve_\alpha$ is
known in quantitative risk management as the
\emph{average-value-at-risk}, also termed the \emph{expected
  shortfall}. It is a coherent risk measure and one of the most widely
used risk measures in the financial industry and in regulation, see \cite{MFE15}.  If
$X\in \Lp[1](\R)$, then $\ve_\alpha(X)$ is finite.  The right-average
quantiles satisfy properties (g1)\,--\,(g5), and subadditivity
(g7) together with (g9). For $\alpha=0$ and $X\in\Lpb[1](\R)$, we
simply obtain the expectation.

The \emph{left-average quantile}, arises from averaging
quantiles in the left tail. It is defined for $\alpha\in(0,1]$ and
$X\in\Lpb[1](\R)$ as
\begin{equation}
  \label{eq:3a}
  \ue_\alpha(X):=\frac{1}{\alpha}
  \int_{0}^{\alpha}\quantile[t]^-(X) dt.
\end{equation}
We have the dual relationship
\[
  \ue_\alpha(X) = -\ve_{1-\alpha}(-X).
\]
The functional $\ue_\alpha(X)$ is a superadditive gauge, satisfying
(g1)\,--\,(g5) and (g8), but not (g9). Measurability of $\ve_\alpha$
and $\ue_\alpha$ follows from \cite{FisslerHolzmann2022}, so that
their conditional versions are well defined.

\paragraph{Expectiles.}

Following \cite{NeweyPowell1987}, the \emph{$\tau$-expectile},
$\tau\in(0,1)$, of a random variable $X\in \Lpb[1](\R)$ is defined as
the unique solution to the equation
\[
  \tau \E\big[(X-z)_+\big] = (1-\tau) \E\big[(X-z)_-\big] 
\]
in $z\in\R$, where $x_+ = \max\{0,x\}$ and $x_- = \max\{0,-x\}$.  For
$\tau=1/2$, we retrieve the usual expectation.  The expectile is
finite on $\Lp[1](\R)$. It has been shown by
\cite{BelliniKlarETAL2014} that expectiles satisfy properties
(g1)\,--\,(g5), and that they are subadditive for $\tau\geq1/2$ and
superadditive for $\tau\le 1/2$. 
If $X=\infty$ with a positive
probability, we let the expectile take the value $\infty$, so that
(g9) holds ensured. The measurability property follows again from
\cite{FisslerHolzmann2022}.



\paragraph{$\Lp$-norm based gauges.}
For $p\in[1,\infty)$, the $\Lp$-norm is defined on $\Lpb(\R)$ as
\begin{displaymath}
  \|X\|_p := \big(\E|X|^p\big)^{1/p}.
\end{displaymath}
This function is not translation equivariant and not necessarily
monotone. It is possible to turn it into a gauge on $\Lpb(\R)$
by letting
\begin{equation}
  \label{eq:norm-gauge}
  \ve^{p,a}(X):=
  \begin{cases}
    \E X+a\big(\E(X-\E X)_+^p\big)^{1/p}, &\text{if } X\in\Lp(\R),\\
    \infty, & \text{if }X=\infty \text{ with positive probability,}
  \end{cases}
\end{equation}
where $a\in[0,1]$. This gauge satisfies (g1)\,--\,(g5), (g6), (g7) and
(g9).  The $\sB(\Mp)/\sB(\Rb)$-measurability property follows as above.

\paragraph{Extensions of gauges.}
For each gauge, we define its maximum and minimum extensions as
\begin{align}
  \label{eq:max-ext}
  \gae^{\vee m}(X) &=\gae\big(\max(X_1,\dots,X_m)\big),\\
  \label{eq:min-ext}
  \gae^{\wedge m}(X) &=\gae\big(\min(X_1,\dots,X_m)\big),
\end{align}
where $X_1,\dots,X_m$ are i.i.d. copies of $X$. The maximum extension
preserves the sublinearity property and the minimum one preserves the
superlinearity property. This construction may
be used to obtain parametric families of gauges as was done by
\cite{mol:tur21} for sublinear expectations. 




\section{Random closed convex sets and related cones}
\label{sec:random-convex-sets}

For a nonempty set $F\subseteq\R^d$, the \emph{support function} is
defined as
\begin{displaymath}
  h(F,u):=\sup\big\{\langle u,x\rangle: x\in F\big\},\quad u\in\R^d.
\end{displaymath}
The support function of the empty set is defined to be $-\infty$. 
A support function can be identified by its subadditivity and positive
homogeneity properties. These properties are summarised by saying that
$h$ is sublinear. Each lower semicontinuous sublinear function is the
support function of a closed convex set, see \cite{schn2}. By
homogeneity, it suffices to restrict the support function onto the
unit sphere $\Sphere$.

Recall that a generic half-space $H_w(t)$ with normal vector
$w\in\R^d$ is defined at \eqref{eq:Hwt}.  Note that
$H_w(\infty)=H_0(t)=\R^d$ for any $w\in\R^d$ and $t\in\R$.  We let
$H_w(-\infty) = \emptyset$.  The support function of $H_w(t)$ is
finite only at $u=cw$ for $w\neq 0$ and $c\geq0$ and then takes the
value $ct$.

For a set $F\subseteq\R^d$ denote its \emph{polar} by
\begin{displaymath}
  F^o:=\big\{u\in\R^d:h(F,u)\leq1\big\}.
\end{displaymath}
If $F=\cone$ is a convex cone, then the polar to $\cone$ is
equivalently defined by
\begin{displaymath}
  \cone^o:=\big\{u\in\R^d: h(\cone,u)\leq 0\big\}.
\end{displaymath}
The \emph{barrier cone} $\bB_F$ of $F$ consists of all $u\in\R^d$ such
that $h(F,u)<\infty$. It is easy to see that $\bB_F$ is indeed a
convex cone and that $F^o\subseteq \bB_F$. If $F=\cone$ is a cone, then
$\bB_F=\cone^o$. 

Denote by $\sF_c^d$ the family of closed convex sets in $\R^d$. A map
$\bX$ from the probability space $(\Omega,\salg,\P)$ to $\sF_c^d$ is
said to be a \emph{random closed convex set} if
$\{\omega: \bX\cap K\neq\emptyset\}\in\salg$ for all compact sets $K$,
see \cite{mo1}. If $\{\omega: \bX\cap K\neq\emptyset\}\in\ssalg$ for a
sub-$\sigma$-algebra $\ssalg\subseteq\salg$ and 
all compact sets $K$, then $\bX$ is said to be $\ssalg$-measurable.

Equivalently, $\bX$ is a random closed convex set if
and only if its support function $h(\bX,u)$ is a random function on
$\R^d$ which may take infinite values if $\bX$ is not bounded. The
only case when the support function takes value $-\infty$ is when
$\bX$ is empty.

By applying the definition of the barrier cone to realisations of
$\bX$, we obtain the barrier cone of $\bX$ 
\begin{displaymath}
  \BX:=\big\{u\in\R^d: h(\bX,u)<\infty\big\}.
\end{displaymath}
If $\bX$ is a.s.\ compact, then $\BX=\R^d$.
Note that $\BX$ is a random set itself, which is not necessarily
closed: its values are so-called $F_\sigma$ sets, being countable
unions of random closed sets $\{u\in\R^d: h(\bX,u)\leq n\}$, $n\geq1$,
see \citet[Lemma~1.8.23]{mo1}. Measurability of $\BX$ is understood
as graph measurability of the map
$\omega\mapsto \bB_{\bX(\omega)}$, see \citet[Sec.~1.3.6]{mo1}.  By
$\sigma(\bX)$ we denote the smallest $\sigma$-algebra which makes
$\bX$ measurable, and by $\sigma(\BX)$ the smallest $\sigma$-algebra
which makes $\BX$ graph measurable.

\begin{example}\label{exmp:solvency cone}
  Important examples of random sets are related to deterministic or
  random cones. In the following we denote a deterministic closed
  convex cone by $\cone$ and a random one by $\bC$. If $\bX=\bC$, then
  $\BX=\bC^o$ and so $\BX$ is a random closed convex set. Let
  $\bX=X+\bC$, where $X$ is a random vector and $\bC$ is a random
  closed convex cone in $\R^d$ which contains $(-\infty,0]^d$. In the
  financial setting, such a cone is called the set of portfolios
  available at price zero, see \cite{kab:saf09}.
  The polar cone $\bC^o$ is said to be a cone of consistent price
  systems.
\end{example}

A random vector $\xi$ is called a \emph{selection} of $\bX$ if
$\xi\in \bX$ almost surely. Let $\Lp(\bX)$ denote the family of
$p$-integrable $\sigma(\bX)$-measurable selections of $\bX$ for
$p\in[1,\infty)$, essentially bounded ones if $p=\infty$, and all
selections if $p=0$. Sometimes, it is convenient to consider
selections which are measurable with respect to a larger
$\sigma$-algebra than the one generated by $\bX$, i.e., such selections may
involve extra randomisation. For instance, a deterministic convex set
may have random selections.

A point $x\in\R^d$ is said to be a \emph{fixed point} of $\bX$ if
$\Prob{x\in\bX}=1$. Clearly a fixed point is a selection of $\bX$. 

If $\Lp(\bX)$ is not empty, then $\bX$ is called
\emph{$p$-integrable}, shortly, \emph{integrable} if $p=1$.
This is the case if $\bX$ is \emph{$p$-integrably bounded}, that is,
\begin{displaymath}
  \|\bX\|:=\sup\big\{\|x\|:\; x\in \bX\big\}
\end{displaymath}
is a $p$-integrable random variable (essentially bounded if
$p=\infty$).
If $p=0$, then the $p$-inegrability of $\bX$ means that
$\bX$ is almost surely nonempty, equivalently, $\Lp[0](\bX)$ is not
empty.

Each a.s.\ nonempty random closed set $\bX$ is the closure of a
countable family of its selections
$\{\xi_n,n\geq1\}\subset\Lp[0](\bX)$. This is called a \emph{Castaing
  representation} of $\bX$, see \cite[Definition~1.3.6]{mo1}. Let us
stress that all members of a Castaing representation are assumed to be
measurable with respect to the $\sigma$-algebra $\sigma(\bX)$
generated by $\bX$.  If $\bX$ is $p$-integrable, then all selections
in its Castaing representation can be also chosen to be
$p$-integrable. A random closed set in $\R^d$ is said to be
\emph{regular closed} if it almost surely coincides with the closure
of its interior. If $\bX$ is regular closed and admits a fixed point,
then it has a Castaing representation which consists of the random
vectors $\xi_n:=u_n\one_{u_n\in\bX}+x\one_{u_n\notin\bX}$, where
$\{u_n,n\geq1\}$ is a countable dense set in $\R^d$ and $x$ is a fixed
point of $\bX$. 

If $\bX$ is integrable, then its \emph{selection expectation} is
defined by 
\begin{equation}
  \label{eq:3exp}
  \E \bX:=\clo\big\{\E \xi:\; \xi\in\Lp[1](\bX)\big\},
\end{equation}
which is the closure of the set of expectations of all integrable
selections of $\bX$, see \cite[Section~2.1.2]{mo1}. If $\bX$ is integrably
bounded, then the closure on the right-hand side is not needed and
$\E \bX$ is compact. This expectation (and its conditional variant)
can be equivalently defined using the support function, see
Section~\ref{sec:gener-expect}, thus providing a dual construction of
the selection expectation. 

If $\bX$ a.s.\ attains compact values, then
\begin{equation}
  \label{eq:4un}
  \bX=\bigcap_{n\geq1} H_{w_n}\big(h(\bX,w_n)\big),
\end{equation}
where $\{w_n,n\geq1\}$ is a countable dense set in $\Sphere$. This
representation may fail if $\bX$ is unbounded with positive
probability, e.g., if $\bX$ is a random half-space. Still, even if
\eqref{eq:4un} fails, \citet[Theorem~3.4]{lep:mol19} establishes that
each random closed convex set satisfies
\begin{equation}
  \label{eq:30}
  \bX=\bigcap_{n\geq1} H_{\eta_n}\big(h(\bX,\eta_n)\big),
\end{equation}
where $\{\eta_n,n\geq1\}\subset\Lp[0](\R^d)$ is a family of
$\salg$-measurable random vectors, whose choice may depend on
$\bX$. By scaling with $\|\eta_n\|$ if $\eta_n\neq 0$ and letting
$H_0(t)=\R^d$ if $\eta_n=0$, it is possible to assume that
$\eta_n\in\Sphere$ a.s. for all $n$. The following result generalises
this by showing that the $\eta_n$'s can be taken to be
$\sigma(\BX)$-measurable and even constants if $\BX$ is regular
closed.

\begin{theorem}
  \label{thr:lep-mol-bis}
  Let $\bX$ be an almost surely nonempty random closed convex set
  in $\R^d$. Then, for each Castaing representation
  $\{W_n,n\geq1\}$ of $\BX\cap\Sphere$,  
  \begin{equation}
    \label{eq:repr}
    \bX=\bigcap_{n\geq1} \Big\{x\in\R^d:\langle W_n,x\rangle\leq
    h(\bX,W_n)\Big\}. 
  \end{equation}
 If $\BX$ is regular closed, then \eqref{eq:repr} holds with
  $W_n=w_n$, where $\{w_n,n\geq1\}$ is a deterministic dense set in
  $\Sphere$.
\end{theorem}
\begin{proof}
  Assume first that $0\in\bX$ a.s. Then
  \begin{displaymath}
    \big\{u\in\R^d: h(\bX^o,u)>0\big\}=\BX\setminus\{0\}.
  \end{displaymath}
  Indeed, for $u\neq0$, we have $h(\bX^o,u)>0$ if and only if
  $h(\bX,u)<\infty$ and so $u\in\BX$. In particular,
  $\bX^o\cap\Sphere=\BX\cap\Sphere$ a.s. Let $\{W_n, n\geq1\}$ be a
  Castaing representation of $\BX\cap\Sphere$, which consists of
  $\sigma(\BX)$-measurable random vectors. Then $\bX^o$ is the closure
  of the union of $\bX^o\cap\{W_nt:t\geq0\}$ for $n\geq1$. For each
  $n\geq1$, let $\{W_n\zeta_{nm},m\geq1\}$ be a Castaing
  representation of $\bX^o\cap\{W_nt:t\geq0\}$.
  Note that $\zeta_{nm}$ is $\salg$-measurable for all $n$ and $m$.

  By the Bipolar theorem, $\bX=(\bX^o)^o$. Hence, $\bX$ is the polar
  set to $\{W_n\zeta_{nm},m,n\geq1\}$, meaning that
  \begin{align*}
    \bX&=\bigcap_{n,m\geq1} \big\{x\in\R^d: \langle
         W_n\zeta_{nm},x\rangle\leq1\big\}\\
       &=\bigcap_{n,m\geq1} \big\{x\in\R^d: \langle
         W_n\zeta_{nm},x\rangle\one_{\zeta_{nm}>0}\leq1\big\}. 
  \end{align*}
  Since $W_n\zeta_{nm}\in\bX^o$ a.s., we have 
  $h(\bX,W_n\zeta_{nm})\leq1$ a.s. Hence,
  \begin{displaymath}
    \bX\supseteq \bigcap_{n,m\geq1} \big\{x\in\R^d: \langle
    W_n\zeta_{nm},x\rangle\leq h(\bX,W_n\zeta_{nm})\big\}.
  \end{displaymath}
  Since $\zeta_{nm}>0$ a.s., 
  \begin{displaymath}
    \bX\supseteq \bigcap_{n\geq1} \big\{x\in\R^d: \langle
    W_n,x\rangle\leq h(\bX,W_n)\big\}.
  \end{displaymath}
  The opposite inclusion is obvious.
  
  Assume now that $\BX$ is regular closed, so that $\BX\cap\Sphere$ is
  regular closed in the relative topology of $\Sphere$. Then
  $W_n:=w_n\zeta_{nm}\one_{w_n\in \BX}\one_{\zeta_{nm}>0}$, $n\geq1$,
  is a Castaing representation of $\bX^o$ and the above argument
  applies.

  If $\bX$ does not necessarily contain the origin, let $\bY=\bX-\xi$
  for an arbitrary $\xi\in\Lp[0](\bX)$. Note that the barrier cone of
  $\bY$ coincides with the barrier cone of $\bX$, so that their
  Castaing representations are identical.  Then
  \begin{align*}
    \bY&=\bigcap_{n\geq1} \big\{y\in\R^d: \langle
    W_n,y\rangle\leq h(\bY,W_n)\big\}\\
    &=\bigcap_{n\geq1} \big\{y\in\R^d: \langle
      W_n,y+\xi\rangle\leq h(\bX,W_n)\big\}\\
    &=\bigcap_{n\geq1} \big\{x\in\R^d: \langle
      W_n,x\rangle\leq h(\bX,W_n)\big\}-\xi.
  \end{align*}
  Finally, note that $x\in\bX$ if and only if $x-\xi\in\bY$.
\end{proof}

If $\bX$ and $\bY$ are two random closed convex sets, then
$\bX\subseteq\bY$ a.s.\ if $h(\bX,W)\leq h(\bY,W)$ a.s. for all
$W\in \Lp[0](\R^d)$, see \citet[Corollary~3.6]{lep:mol19}, or
equivalently, by rescaling for all $W\in \Lp[0](\Sphere)$.  The
following result shows that in some cases it is possible to reduce the
choice of $W$ to obtain the same characterisation.
  
\begin{lemma}
  \label{lemma:subset}
  Let $\bX$ and $\bY$ be two random closed convex sets. Assume that
  $\bY$ is a random compact convex set or that $\BY$ is regular closed.
  Then $\bX\subseteq\bY$ a.s.\ if and only if $h(\bX,w_n)\leq h(\bY,w_n)$
  a.s. for any countable dense set $\{w_n,n\geq1\}\subset\Sphere$.
\end{lemma}
\begin{proof}
  Under the imposed conditions,  \eqref{eq:4un} holds for $\bY$ and implies that
  \begin{align*}
    \qquad\qquad
    \bY&=\bigcap_{n\geq1} \big\{x\in\R^d:\langle w_n,x\rangle\leq
    h(\bY,w_n)\big\}\\
    &\supseteq \bigcap_{n\geq1} \big\{x\in\R^d:\langle w_n,x\rangle\leq
    h(\bX,w_n)\big\}\supseteq \bX\quad \text{a.s.}\qquad\qquad \qedhere
  \end{align*}
\end{proof}


\section{Conditional set-valued gauge functions}
\label{sec:cond-set-valu}

Fix a $p\in\{0\}\cup [1,\infty]$.  Let $\bX$ be a $p$-integrable
random closed convex set, which is then almost surely nonempty.  For
each $W\in\Lp[0](\R^d)$, the support function $h(\bX,W)$ is a random
variable with values in $(-\infty,\infty]$, see
\citet[Lemma~3.1]{lep:mol19}. While $h(\bX,W)$ is not necessarily
integrable, its negative part is always integrable if $\bX$ is
$p$-integrable and $W\in\Lp[q](\R^d)$, where $1/p+1/q=1$. Indeed,
choose any $\xi\in\Lp(\bX)$, and write
\begin{displaymath}
  h(\bX,W)=h(\bX-\xi,W)+\langle \xi,W\rangle.
\end{displaymath}
The second summand on the right-hand side is integrable, while the
first one is nonnegative. Thus, $h(\bX,W)\in\Lpb[1](\R)$.

Consider a gauge function $\gae$ satisfying the conditions of Proposition \ref{prop:conditional-g}.  For each $W\in\Lp[\infty](\R^d)$,
the random variable $h(\bX,W)$ belongs to $\Lpb(\R)$, so that its
conditional gauge $\gae(h(\bX,W)|\ssalg)$ is well defined (up to an
almost sure equivalence).
Then $H_W\big(\gae(h(\bX,W)|\ssalg)\big)$ is a random half-space, which
becomes the whole space if $\gae(h(\bX,W)|\ssalg)=\infty$, and is empty
if $\gae(h(\bX,W)|\ssalg)=-\infty$. 
Note that this half-space does not change if $W$ is scaled, that is,
if $W$ is replaced by $\gamma W$ for a random variable
$\gamma\in\Lp[\infty]\big((0,\infty);\ssalg\big)$.

\begin{theorem}
  \label{thr:maximal}
  Fix a $p\in\{0\}\cup[1,\infty]$.  Let $\bX$ be a $p$-integrable
  random closed convex set.
  Let $\ssalg$ be a sub-$\sigma$-algebra of $\salg$.  There exists the
  largest (in the inclusion order) $\ssalg$-measurable random closed
  convex set $\bY$ such that
  \begin{equation}
    \label{eq:6}
    h(\bY,W)\leq \gae(h(\bX,W)|\ssalg) \quad\text{a.s.}
  \end{equation}
  for all $W\in \Lp[\infty](\R^d;\ssalg)$.
\end{theorem}
\begin{proof}
  Consider
  \begin{equation}
    \label{eq:5}
    \bY':=\bigcap_{W\in \Lp[\infty](\R^d;\ssalg)}
    H_W\big(\gae(h(\bX,W)|\ssalg)\big).
  \end{equation}
  By construction, $\bY'$ has closed realisations, so it is a random
  closed set in a $\sigma$-algebra chosen to be suitably rich. By
  \citet[Lemma~4.3]{lep:mol19}, there exists the largest
  $\ssalg$-measurable random closed set $\bY$ such that
  $\bY\subseteq\bY'$ almost surely (it is called the measurable
  version of $\bY'$). Recall that the probability space is assumed
  to be complete.

  By construction, $\bY$ satisfies \eqref{eq:6}. Assume that $\bY$ is
  not the largest set which satisfies \eqref{eq:6}. Then there exists
  an $\ssalg$-measurable random vector $\xi$ such that
  $\xi\notin\bY$ with a positive probability and $\bZ=\bY\cup\{\xi\}$
  satisfies \eqref{eq:6}. However, then $\bY'\cup\{\xi\}$ is a subset
  of the right-hand side of \eqref{eq:5}. Hence, $\bY\cup\{\xi\}$ is an
  $\ssalg$-measurable subset of $\bY'$, which is a contradiction
  unless $\xi\in\bY$ a.s.
\end{proof}

The set $\bY$ from Theorem~\ref{thr:maximal} is denoted as
$\Gae(\bX|\ssalg)$ and is called a \emph{conditional set-valued gauge}
of $\bX$ given $\ssalg$.  The function $\Gae(\cdot|\ssalg)$ is said to
be the set-valued extension of the (scalar) conditional gauge
$\gae(\cdot|\ssalg)$. We replace in our notation $\ssalg$ by a
random variable, if $\ssalg$ is generated by it.

\begin{lemma}
  Condition \eqref{eq:6} can equivalently be imposed for all
  $W\in\Lp[0](\Sphere;\ssalg)$ or for all $W\in\Lp[0](\R^d;\ssalg)$.
\end{lemma}
\begin{proof}
  If $W\in\Lp[\infty](\R^d;\ssalg)$, it can be represented as the product
  of $\|W\|$ and a random vector $W'$ such that $W'\neq0$ a.s. For
  this, one may need to redefine $W$ in a measurable way outside of
  the event $\{W=0\}$. It remains to notice that it is possible to
  take $\|W\|$ as a factor out of the both sides of \eqref{eq:6} due
  to the homogeneity property, see
  Proposition~\ref{prop:properties}(ii). 
\end{proof}


\begin{remark}
  \label{rem:lm}
  The conditional translation equivariance property (see
  Proposition~\ref{prop:properties}(v)) implies that \eqref{eq:6} can
  be equivalently written as
  \begin{displaymath}
    g\big(h(\bX,W)-h(\bY,W)|\ssalg\big)\geq 0\quad \text{a.s.}
  \end{displaymath}
  for all $W\in \Lp[\infty](\R^d;\ssalg)$ and \eqref{eq:5} as
  \begin{equation}
    \label{eq:larger-intersection}
    \bY'=\bigcap_{W\in \Lp[\infty](\R^d;\ssalg)}
    \Big\{x\in\R^d: \gae\big(h(\bX,W)-\langle x,W\rangle|\ssalg\big)\geq 0\;
    \text{a.s.}\Big\}.
  \end{equation}
  The latter expression yields an $\ssalg$-measurable random closed
  convex set even for intersections taken over all $W$, which
  belong to $\Lp[0](\R^d;\salg)$, that is, for those, which are not
  necessarily measurable with respect to $\ssalg$. This approach was
  pursued by \cite{lep:mol19} for the gauges being the expectation,
  essential supremum and essential infimum. Computing the intersection
  in \eqref{eq:larger-intersection} over all (not necessarily
  $\ssalg$-measurable) $W$ may be complicated in case of further gauges.  
  Moreover, it may result in smaller sets, as the following example illustrates.
\end{remark}


\begin{example}
  \label{ex:half-space}
  Let $\bX=H_U(0)$ be a half-space with unit outer normal
  $U\in\Lp[0](\Sphere)$ such that the distribution of $U$ is
  non-atomic. Note that in this case $\BX$ is a line $\{tU: t\geq0\}$,
  which is not regular closed, in particular, Lemma~\ref{lemma:subset}
  is not applicable. For any fixed direction
  $w\in\Sphere$, the support function $h(\bX,w)$ is infinite almost
  surely, so that $\gae(h(\bX,w))$ is infinite too. Hence, the
  intersection of $H_w\big(\gae(h(\bX,w)|\ssalg)\big)$ over all
  $w\in\Sphere$ is whole space. If $\ssalg$ is generated by $U$, then
  the right-hand side of \eqref{eq:5} is nontrivial and equal to
  $H_U(0)$.
  
  However, if $\ssalg$ is independent of $U$, then $h(\bX,W)=\infty$
  a.s.\ for all $\ssalg$-measurable $W$. In this case,
  $\gae(h(\bX,W)|\ssalg)=\infty$ a.s.\ and so $\Gae(\bX|\ssalg)$
  becomes the whole space. However, taking the intersection in
  \eqref{eq:larger-intersection} over all $\salg$-measurable directon
  vectors yields the random set
  \begin{align*}
    \bigcap_{V\in \Lp[\infty](\R^d;\salg)}
    &\Big\{x\in\R^d: \gae\big(h(\bX,V)-\langle x,V\rangle|\ssalg\big)\geq 0\;
      \text{a.s.}\Big\}\\
    &\subset \big\{x\in\R^d: \gae\big(h(\bX,U)-\langle x,U\rangle|\ssalg\big)\geq 0\;
     \text{a.s.}\big\}\\
    &=\{x\in\R^d: -\langle x,U\rangle\geq 0\;
      \text{a.s.}\big\}=H_U(0)=\bX,
   \end{align*}
   which is not the whole space.
\end{example}

\begin{theorem}\label{theorem:properties Gae}
  Fix a $p\in\{0\}\cup[1,\infty]$.  The conditional gauge
  $\Gae(\cdot|\ssalg)$ is a map from the family of almost
  surely nonempty ($p$-integrable) random closed convex sets to the
  family of $\ssalg$-measurable random closed convex sets which is
  \begin{enumerate}[(G1)]
  \item law-determined, that is, $\Gae(\bX|\ssalg)=\Gae(\bX'|\ssalg)$
    if $\bX$ and $\bX'$ have the same conditional distribution given
    $\ssalg$;
  \item constant preserving: $\Gae(\bX|\ssalg)=\bX$ a.s.\ if $\bX$ is 
    $\ssalg$-measurable;
  \item positively homogeneous:
    $\Gae(\Gamma\bX|\ssalg)=\Gamma\Gae(\bX|\ssalg)$ a.s.\ for all
    invertible $d\times d$ matrices $\Gamma$ with entries from
    $\Lp[0]([0,\infty);\ssalg)$;
  \item monotone:
    $\Gae(\bX|\ssalg)\subseteq \Gae(\bX'|\ssalg)$ a.s.\ if
    $\bX\subseteq\bX'$ a.s.\ for two $p$-integrable random closed sets
    $\bX$ and $\bX'$;
  \item translation equivariant:
    $\Gae(\bX+\bZ|\ssalg)=\Gae(\bX|\ssalg)+\bZ$ for all
    $p$-integrable random closed convex sets $\bZ$ if $\Gae(\bX |\ssalg)$ is
    not empty.
  \end{enumerate}
\end{theorem}
\begin{proof}
  \begin{enumerate}[(G1)]
  \item Follows from the fact that the conditional
    gauges are law-determined and that $h(\bX,W)$ and $h(\bX',W)$ share
    the same distribution for any $W\in\Lp[\infty](\R^d;\ssalg)$.
  \item If $\bX$ is $\ssalg$-measurable, then for
    $W\in\Lp[\infty](\R^d;\ssalg)$ it holds that
    $\gae(h(\bX,W)|\ssalg)=h(\bX,W)$, so that the set in \eqref{eq:5}
    equals $\bX$ a.s. 
  \item Note that $h(\Gamma\bX,W)=h(\bX,\Gamma^\top W)$ for any
    $W\in\Lp[\infty](\R^d;\ssalg)$. Thus, for
    $\bY = \Gae(\Gamma \bX|\ssalg)$ it holds that
    $h(\bY,W)\leq \gae(h(\Gamma\bX,W)|\ssalg)$ if and only if
    \begin{displaymath}
      h(\Gamma^{-1}\bY,V)=h(\bY,(\Gamma^{-1})^\top V)
      \leq \gae(h(\bX,V)|\ssalg)
    \end{displaymath}
    for $V=\Gamma^\top W$. Since $\Gamma$ is invertible, the sets
    $\Lp[0](\R^d;\ssalg)$ and its image under $\Gamma$ coincide. Thus,
    $\Gamma^{-1}\Gae(\Gamma\bX|\ssalg)=\Gae(\bX|\ssalg)$. 
  \item Follows from the fact that the set constructed by the
    right-hand side of \eqref{eq:5} for $\bX$ is a subset of the one
    constructed by $\bX'$.
  \item Follows from the fact that for any $W\in\Lp[\infty](\R^d;\ssalg)$
    \begin{displaymath}
      \gae(h(\bX+\bZ,W)|\ssalg)
      =\gae(h(\bX,W)+h(\bZ,W)|\ssalg)
      =\gae(h(\bX,W)|\ssalg)+h(\bZ,W). 
    \end{displaymath}
  \end{enumerate}
\end{proof}

\begin{remark}[Unconditional set-valued gauge]
  If $\ssalg$ is trivial, then \eqref{eq:6} becomes
  $h(\bY,W)\leq \gae(h(\bX,W))$ for all deterministic $W=w\in\Sphere$
  and so the unconditional gauge is given by
  \begin{equation}
    \label{eq:unconditional}
    \Gae(\bX)=\bigcap_{w\in\Sphere} H_w\big(\gae(h(\bX,w))\big).
  \end{equation}
  Since the right-hand side is deterministic, $\Gae(\bX)$ is a deterministic
  closed convex set. Such a set is called the \emph{Wulff shape}
  associated with the function $\gae(h(\bX,w))$, $w\in\Sphere$, see
  \citet[Section~7.5]{schn2}. If $\gae$ is sublinear and
  $\gae(h(\bX,w))$ is lower semicontinuous in $w$, then we have
  $h(\Gae(\bX),w)=\gae(h(\bX,w))$ for all $w\in\Sphere$. In this case $\bY$ becomes
  the sublinear expectation of $\bX$, studied by \cite{mol:mueh21}.
\end{remark}




\begin{proposition}\label{prop:superlinear}
  If the conditional gauge $\gae(\cdot|\ssalg)$ is superadditive, then its set-valued
  extension $\Gae(\cdot|\ssalg)$ satisfies 
  \begin{displaymath}
    \Gae(\bX'+\bX''|\ssalg)\supseteq \Gae(\bX'|\ssalg)+\Gae(\bX''|\ssalg).
  \end{displaymath}
\end{proposition}
\begin{proof}
  Superadditivity of $\Gae(\cdot|\ssalg)$ follows from the
  fact that for all $W\in \Lp[\infty](\R^d;\ssalg)$
  \begin{align*}
    \gae(h(\bX'+\bX'',W)|\ssalg)&=\gae(h(\bX',W)+h(\bX'',W)|\ssalg)\\
    &\geq \gae(h(\bX',W)|\ssalg)+\gae(h(\bX'',W)|\ssalg). 
  \end{align*}
  Thus, $\Gae(\bX'+\bX''|\ssalg)$ contains the sum of the largest
  random closed convex sets whose support functions are dominated by
  $\gae(h(\bX',W)|\ssalg)$ and $\gae(h(\bX'',W)|\ssalg)$, which are
  $\Gae(\bX'|\ssalg)$ and $\Gae(\bX''|\ssalg)$, respectively. 
\end{proof}

\begin{remark}[Set-valued utility]
  Proposition \ref{prop:superlinear} implies that if
  $\gae(\cdot|\ssalg)$ is superadditive, $\Gae(\cdot|\ssalg)$ becomes
  a conditional set-valued utility functions, which can be used as an
  acceptance criterion for set-valued portfolios. Namely, $\bX$ is
  said to be conditionally acceptable if $0\in\Gae(\bX|\ssalg)$. Then,
  $0\in\Gae(\bX'|\ssalg)$ and $0\in\Gae(\bX''|\ssalg)$ for two random
  closed convex sets $\bX'$ and $\bX''$ imply that
  $0\in\Gae(\bX'+\bX''|\ssalg)$, meaning that $\bX'+\bX''$ is also
  conditionally acceptable.
\end{remark}

The following examples work for any choice of the gauge function
$\gae$, which admits a conditional version.  

\begin{example}
  Assume that $\bX:=\xi \bZ$ and $\xi\geq0$ is independent of
  $\sigma(\bZ)$. Then
  \begin{displaymath}
    \gae(h(\bX,W)|\bZ)=\gae(h(\bZ,W)\xi|\bZ)
    =\gae(\xi|\bZ)h(\bZ,W).
  \end{displaymath}
  Thus, $\Gae(\bX|\bZ)=\gae(\xi|\bZ)\bZ$.
\end{example}

\begin{example}
  Let $\bX:=(-\infty,V_1]\times(-\infty,V_2]$ be the quadrant in
  $\R^2$ with its upper right corner at $(V_1,V_2)$. Let $\ssalg$ be
  generated by $V_2$. If $w=(w_1,w_2)\in\R_+^2$, then
  \begin{displaymath}
    h(\bX,w)=w_1V_1+w_2V_2,
  \end{displaymath}
  and the support function is infinite for all other $w\in \R^2$.  For
  any gauge $\gae$, since $W\in\Lp[\infty](\R_+^2;\ssalg)$,
  \begin{displaymath}
    \gae(h(\bX,W)|V_2)=W_1\gae(V_1|V_2)+W_2V_2.
  \end{displaymath}
  Thus, 
  $\Gae(\bX|V_2)=\big(-\infty,\gae(V_1|V_2)\big]\times\big(-\infty,V_2\big]$.
\end{example}

\begin{example}
  Let $\bX:=[0,V_1]\times[0,V_2]$ be the rectangle in $\R^d$ with its
  upper right corner at $(V_1,V_2)$. Let $\ssalg$ be generated by
  $V_2$. If $W=(W_1,W_2)$, then
  \begin{displaymath}
    h(\bX,W)=
    \begin{cases}
      0, & W_1\leq 0, W_2\leq 0,\\
      W_1V_1+W_2V_2 , & W_1>0, W_2>0,\\
      W_1V_1 & W_1>0, W_2\leq 0,\\
      W_2V_2 & W_1\leq 0, W_2>0.
    \end{cases} 
  \end{displaymath}
  Since $W$ is assumed to be $\ssalg$-measurable,
  \begin{displaymath}
    \gae(h(\bX,W)|\ssalg)=
    \begin{cases}
      0, & W_1\leq 0, W_2\leq 0,\\
      W_1 \gae(V_1|V_2)+W_2V_2 , & W_1>0, W_2>0,\\
      W_1\gae(V_1) & W_1>0, W_2\leq 0,\\
      W_2V_2 & W_1\leq 0, W_2>0.
    \end{cases} 
  \end{displaymath}
  Thus, $\Gae(\bX|V_2)$ is the rectangle
  $\big[0,\gae(V_1|V_2)\big]\times\big[0,V_2\big]$.
\end{example}

\section{Special cases of set-valued gauges}
\label{sec:special-cases-gauges}

While we write $\Gae(\bX|\ssalg)$ for a generic conditional set-valued
gauge, we use special notation for the most important gauge functions
to denote their set-valued variants. For example,
$\ve_\alpha(\bX|\ssalg)$ denotes the set-valued gauge constructed from
the right-average quantile $\ve_\alpha$.

\subsection{Generalised conditional expectation}
\label{sec:gener-expect}

Let $\gae(\cdot|\ssalg)$ be the generalised conditional expectation,
which is defined on $\Lpb[1](\R)$. Then $\Gae(\bX|\ssalg)$ is the
generalised conditional expectation of $\bX$ given $\ssalg$, see
\citet[Definition~2.1.79]{mo1}. Indeed, the generalised conditional
expectation $\E(\bX|\ssalg)$ satisfies
\begin{equation}
  \label{eq:2}
  h\big(\E(\bX|\ssalg),W\big)=\E\big(h(\bX,W)|\ssalg\big)\quad\text{a.s.}
\end{equation}
for all $\ssalg$-measurable $W$.  For deterministic $W$, this fact is
well known (see, e.g., \citet[Theorem~2.1.72]{mo1}), for random $W$ it is
proved in \citet[Lemma~6.7]{lep:mol19}. Thus, the generalised
conditional expectation is the largest $\ssalg$-measurable random
closed set $\bY$ which satisfies $h(\bY,W)\leq\E(h(\bX,W)|\ssalg)$
a.s.\ for all $W\in \Lp[\infty](\R^d;\ssalg)$, meaning that
\eqref{eq:6} is satisfied with equality.  If the left-hand side of
\eqref{eq:2} is finite, then necessarily $W\in\BX$. Hence,
$\bB_{\E(\bX|\ssalg)}\subseteq\BX$. The inverse inclusion does not
always hold.

If we consider gauges given by the maximum and minimum extensions of
the expectation defined at \eqref{eq:max-ext} and \eqref{eq:min-ext},
then we obtain the set-valued gauges given by
$\E\big(\bX_1\cup\cdots\cup\bX_m|\ssalg\big)$ and
$\E\big(\bX_1\cap\cdots\cap\bX_m|\ssalg\big)$, respectively, where
$\bX_1,\dots,\bX_m$ are i.i.d.\ copies of $\bX$. The 
intersection-based gauge is nonempty only if
$\bX_1\cap\cdots\cap\bX_m\neq\emptyset$ with probability one.

\subsection{Conditional quantiles}
\label{sec:cond-quant}

Assume now that $\gae=\quantile^-$ is the lower
$\alpha$-quantile. Then $\quantile^-(\bX|\ssalg)$ is the largest closed
convex set such that its support function is dominated by
$\quantile^-(h(\bX,W))$ for all $W\in\Lp[0](\Sphere;\ssalg)$.

The set-valued extension $\essinf(\bX|\ssalg)$ of the conditional
essential infimum is the largest $\ssalg$-measurable random closed
convex set $\bY$ such that
\begin{displaymath}
  h(\bY,W)\leq \essinf\big(h(\bX,W)|\ssalg\big)\quad \text{a.s.}
\end{displaymath}
for all $W\in \Lp[\infty](\R^d;\ssalg)$. A similar concept is the
\emph{conditional core} $\core(\bX|\ssalg)$ of $\bX$, which is defined
by \cite{lep:mol19} as the largest $\ssalg$-measurable random closed
convex subset of $\bX$, that is, $h(\core(\bX|\ssalg),V)\leq h(\bX,V)$
for all $\salg$-measurable $V$.  By Theorem~4.12 from \cite{lep:mol19},
\begin{displaymath}
  \core(\bX|\ssalg)=\bigcap_{V\in \Lp[\infty](\R^d)}
  H_V\big(\essinf(h(\bX,V)|\ssalg)\big),
\end{displaymath}
where the intersection is taken over all $\salg$-measurable $V$, not
necessarily those which are $\ssalg$-measurable, cf.\ 
\eqref{eq:larger-intersection} where the intersection is taken over
all $W\in\Lp[\infty](\R^d;\ssalg)$. Thus,
\begin{displaymath}
  \core(\bX|\ssalg)\subseteq \essinf(\bX|\ssalg),
\end{displaymath}
and the inclusion may be strict as Example~\ref{ex:half-space} shows.
This inclusion shows that $\essinf(\bX|\ssalg)$ is not empty if the
conditional core is not empty, which is the case if $\bX$ admits an
$\ssalg$-measurable selection.

In the unconditional case, $\core(\bX)$ is the set of fixed points of
$\bX$. Still, $\essinf(\bX)$ may be different: for example, if $\bX$
is the random half-space from Example~\ref{ex:half-space}, then
$\essinf h(\bX,u)=\infty$ for all deterministic $u$ and so
$\essinf(\bX)=\R^d$, while $\core(\bX)=\{0\}$.  On the other hand, by
Lemma~\ref{lemma:subset}, we have $\core(\bX)=\essinf(\bX)$ if $\BX$
is regular closed. Indeed, then $\core(\bX)$ is the largest
determinist set such that $h(\core(\bX),w)\leq h(\bX,w)$ a.s.\ for all
$w\in\Sphere$, which is the same as
$h(\core(\bX),w)\leq \essinf h(\bX,w)$. The latter inequality
identifies $\essinf(\bX)$.

The \emph{conditional convex hull} $\chull(\bX|\ssalg)$ of $\bX$, is
the smallest (in the sense of inclusion) $\ssalg$-measurable random
closed convex set which a.s.\ contains $\bX$, see
\citet[Definition~5.1]{lep:mol19}.  It follows from
\citet[Theorem~5.2]{lep:mol19} that the support function of
$\chull(\bX|\ssalg)$ equals $\esssup h(\bX,W)$ for all
$W\in\Lp[0](\Sphere;\ssalg)$. If $\gae$ is the essential supremum,
then the right-hand side of \eqref{eq:5} equals $\chull(\bX|\ssalg)$,
meaning that
\begin{equation}
  \label{eq:chull-cond}
  \chull(\bX|\ssalg)=\esssup(\bX|\ssalg).
\end{equation}
This equality is due to the fact that the essential supremum is
sublinear.

For a general gauge, we always have
\begin{displaymath}
  \essinf(\bX|\ssalg)\subseteq\Gae(\bX|\ssalg)
  \subseteq \esssup(\bX|\ssalg).
\end{displaymath}

\subsection{Average quantiles}
\label{sec:integrated-quantiles}

The conditional right-average quantile $\ve_\alpha(\cdot|\ssalg)$ is
subadditive. Then $\ve_\alpha(h(\bX,W)|\ssalg)$ is a sublinear function of
$W\in \Lp[\infty](\R^d;\ssalg)$.  


The conditional left-average quantile $\ue_\alpha$ is superlinear. If
$\bY=\ue_\alpha(\bX|\ssalg)$, then its
width in direction $w\in\Sphere$ is bounded from above by 
\begin{align*}
  h(\bY,w)+h(\bY,-w)
  &\leq \ue_\alpha(h(\bX,w)|\ssalg)+\ue_\alpha(h(\bX,-w)|\ssalg)\\
  &\leq \ue_\alpha(h(\bX,w)+h(\bX,-w)|\ssalg).
\end{align*}
Thus, the width of $\bY$ is at most the conditional
left-average quantile of the width of $\bX$. 

\section{Random singletons and conditional depth-trimmed regions}
\label{sec:random-singletons}


Let $X\in \Lpb(\R^d)$.
By definition, $\bY=\Gae(\{X\}|\ssalg)$ is the largest
$\ssalg$-measurable random closed convex set such that
\begin{displaymath}
  h(\bY,W)\leq \gae\big(\langle X,W\rangle|\ssalg\big)
\end{displaymath}
for all $W\in\Lp[\infty](\R^d;\ssalg)$.

If $\gae$ is the expectation and $X\in\Lp[1](\R^d)$, then
$\Gae(\{X\}|\ssalg)=\big\{\E(X|\ssalg)\big\}$.

\begin{example}[Unconditional setting with sublinear gauges]
  Fix a $p\in[1,\infty]$ and assume that $\gae:\Lpb[p](\R)\to\Rb$ is
  subadditive and lower semicontinuous. If $w_n\to w$, then $\langle
  X,w_n\rangle$ converges to $\langle X,w\rangle$ in
  $\sigma(\Lp[p],\Lp[q])$, so that the function
  $\gae(\langle X,w\rangle)$ is lower semicontinuous. Since $\gae$ is
  subadditive, this function is sublinear in $w$. Indeed, it is clearly
  homogeneous and
  \begin{displaymath}
    \gae\big(\langle X,w+v\rangle\big)
    =\gae\big(\langle X,w\rangle+\langle X,v\rangle\big)
    \leq \gae\big(\langle X,w\rangle\big)
    +\gae\big(\langle X,v\rangle\big).
  \end{displaymath}
  Thus, there exists a convex body $\Gae(\{X\})$ such that
  $h(\Gae(\{X\}),w)=\gae(\langle X,w\rangle)$ for all $w\in\R^d$, see
  \cite{schn2}. This way to associate multivariate distributions with
  convex bodies was suggested by \cite{mol:tur21}, where examples and
  properties of this construction can be found.
\end{example}

For a random vector $X=(X_1,\dots,X_d)\in\Lpb(\R^d)$, we denote
\begin{displaymath}
  \gae(X|\ssalg)=\big(\gae(X_1|\ssalg),\dots, \gae(X_d|\ssalg)\big).
\end{displaymath}

\begin{proposition}
  \label{prop:singelton}
  Let $\bX=\{X\}$ be a singleton with
  $X=(X_1,\dots,X_d)\in\Lpb(\R^d)$. Then the following holds.
  \begin{enumerate}[(i)]
  \item If the gauge $\gae$ is superadditive, then
    $\Gae(\{X\}|\ssalg)\subseteq \{\gae(X|\ssalg)\}$ almost surely. 
  \item If the gauge $\gae$ is subadditive, then
    $\Gae(\{X\}|\ssalg)\subseteq\gae(X|\ssalg)+\R_-^d$ a.s.
  \end{enumerate}
\end{proposition}
\begin{proof}
  (i) If the conditional gauge is superadditive, then the width of
  $\bY=\Gae(\{X\}|\ssalg)$ in direction $W$ is given by
  \begin{align*}
    h(\bY,W)+h(\bY,-W)&\leq \gae(\langle X,W\rangle|\ssalg)
    +\gae(\langle X,-W\rangle|\ssalg)\\
    &\leq \gae(0|\ssalg)=0
  \end{align*}
  for all $W\in\Lp[\infty](\R^d;\ssalg)$. Thus, $\Gae(\{X\}|\ssalg)$
  is a.s.\ a singleton $\{Y\}$ or it is empty.  Assuming the former,
  the components $Y_i$ of $Y= (Y_1,\ldots,Y_d)$ are given by
  $\gae(X_i|\ssalg)$.

  (ii) If the conditional gauge is subadditive, then
  \begin{displaymath}
    \gae(\langle X,W\rangle|\ssalg)\leq \langle \gae(X|\ssalg),W\rangle.
  \end{displaymath}
  for all $W\in\Lp[\infty](\R^d_+;\ssalg)$. Note here that the
  non-negativity of the components of $W$ is essential to apply the
  subadditivity property.  Then
  \begin{displaymath}
    h(\bY,W)\leq \langle \gae(X|\ssalg),W\rangle
  \end{displaymath}
  for all $W\in\Lp[\infty](\R^d_+;\ssalg)$, so that $\bY\subseteq
  \gae(X|W)+\R_-^d$. 
\end{proof}

The proof for the superadditive case actually implies that
$\Gae(\{X\}|\ssalg)$ is empty if
\begin{displaymath}
  \gae(\langle X,W\rangle|\ssalg)
  +\gae(\langle X,-W\rangle|\ssalg)<0 
\end{displaymath}
for at least one $W\in\Lp[\infty](\R^d;\ssalg)$. 

\begin{example}[Conditioning on one component]
  Let $X=(X_1,X_2)$ be a random vector in $\R^2$, and let $\ssalg$ be
  generated by $X_2$. Then
  \begin{align*}
    \bigcap_{W\in \Lp[\infty](\R^2;\ssalg)}
    H_W(\gae(\langle X,W\rangle|X_2))
    &=\bigcap_{W\in \Lp[\infty](\R^2;\ssalg)}
    H_W(\gae(X_1W_1|X_2)+\langle (0,X_2),W\rangle)\\
    &=(0,X_2)+\bigcap_{W\in \Lp[\infty](\R^2;\ssalg)}
    H_W(\gae(X_1W_1|X_2)). 
  \end{align*}
  If $W_1\geq 0$ a.s., it is possible to use the homogeneity property
  to obtain
  \begin{align*}
    \Gae(\{X\}|\ssalg)
    &\subset
      (0,X_2)+\bigcap_{W\in \Lp[\infty](\R_+\times\R;\ssalg)}
      H_W(\langle (\gae(X_1|X_2),0),W\rangle)\\
    &=(\gae(X_1|X_2),X_2)+(-\infty,0]\times\{0\}. 
  \end{align*}
  Note here that the polar to $\R_+\times\R$ is
  $(-\infty,0]\times\{0\}$. If $W_1\leq 0$ a.s., then
  \begin{displaymath}
    \gae(X_1W_1|X_2)=(-W_1)\gae(-X_1|X_2).
  \end{displaymath}
  Thus,
  \begin{displaymath}
    \Gae(\{X\}|X_2)\subseteq
    [-\gae(-X_1|X_2),\gae(X_1|X_2)]\times\{X_2\}.
  \end{displaymath}
  Note that the first component is bounded between the conditional
  scalar gauge and its dual.
\end{example}

Consider now several particular gauges applied to random singletons. 

\begin{example}[Conditional half-space depth]
  Consider the gauge given by a quantile $\gae=\quantile^-$. It
  is defined for all $X\in\Lpb[0](\R^d)$. In the unconditional
  setting, $\Gae(\{X\})$ is the Tukey (or half-space) depth-trimmed
  region at level $1-\alpha$, see \cite{tuk75} and
  \cite{nag:sch:wer19}. If $X$ is uniformly distributed in a convex
  set $K$, the set $\Gae(\{X\})$ is called the floating body of $K$,
  see \cite{nag:sch:wer19}. It is known from \cite{bob10} that, if the
  distribution of $X$ is log-concave and $\alpha\in(1/2,1)$, then
  \begin{displaymath}
    h(\Gae(\{X\}),u)=\quantile^-(\langle X,u\rangle),
    \qquad u\in \Sphere. 
  \end{displaymath}
  If $\alpha=1$, then $\gae$ is the essential supremum,  and
  $\esssup(\{X\})$ is the convex hull of the support of $X$.

  The conditional variant of the half-space depth may be used in a
  multiple output regression setting to introduce a notion of depth
  (for the responses) conditioned on the value of the regressors, as
  indicated in \cite{hallin10} and \cite{MR2420242}.
\end{example}

\begin{example}
  Assume that $X$ has a spherically symmetric distribution, that is,
  the distribution of $X$ is invariant under orthogonal
  transformations, see \citet[Section~2.1]{fan:kot:ng90}. It is known
  (see Theorem~2.4 ibid) that this holds if and only if for all
  $w\in\Sphere$ the projection $\langle X,w\rangle$ has the same distribution as
  $\|w\|X_1$, where $X_1$ is the first component of $X$. If
  $\gae(X_1)\ge 0$, then
  \begin{align*}
    \Gae(\{X\})
    &= \bigcap_{w\in\Sphere} \big\{x\in\R^d :
      \langle x,w\rangle \le \gae(\langle X,w\rangle)\big\}\\
    &= \bigcap_{w\in\Sphere} \big\{x\in\R^d :
      \langle x,w\rangle \le \|w\|\gae(X_1)\big\}
    = \gae(X_1) B\,,
  \end{align*}
  where $B$ is the unit Euclidean ball. If $Y=\mu+\Gamma X$ for a deterministic location $\mu\in\R^d$ an
  invertible scale matrix $\Gamma\in\R^{d\times d}$, then (G2) and (G3)
  yield that $\Gae(\{Y\})= \mu+\gae(X_1) \Gamma B$, which is a
  translated ellipsoid. For instance, if $Y$ follows the centred
  normal distribution with covariance matrix $\Sigma$, then
  $\Gamma=\Sigma^{1/2}$ and $X_1$ is the standard
  normal. Conveniently, a closed form solution for the gauge of a
  one-dimensional standard normal is known for many important
  cases. E.g., for the (upper and lower) quantile
  $\quantile^-(X_1) = \Phi^{-1}(\alpha)$, $\alpha\in(0,1)$, and for
  average quantiles
  \begin{equation}\label{eq:normal distribution}
    \ve_\alpha(X_1) = (1-\alpha)^{-1}\varphi(\Phi^{-1}(\alpha)),
    \quad\quad
    \ue_\alpha(X_1) = -\alpha^{-1}\varphi(\Phi^{-1}(\alpha)),
  \end{equation}
  where $\Phi^{-1}$ is the quantile function of a standard normal and
  $\varphi$ its density function. 

  Finally, for the conditional setup, assume that $X$ given $\ssalg$
  follows a normal distribution with conditional mean vector $\mu$ and invertible
  conditional covariance matrix $\Sigma$.  Such a situation
  arises, e.g., when $X$ is jointly normal and $\ssalg$ is generated
  by a subvector of $X$. Another relevant example is a mean--normal
  variance mixture such as a GARCH-process with Gaussian innovations,
  which is a relevant model class in quantitative risk management; see
  \citet[Chapter 6.2]{MFE15} and references therein.  Then, the above
  arguments work similarly, and we obtain
  \[
    \Gae(\{X\}|\ssalg) =\mu+ \gae(X_1)\Sigma^{1/2}B\,.
  \]
  That is, the conditional gauge is a random ellipsoid obtained by
  multiplying the Euclidean ball with $\gae(X_1)$ and the random
  matrix $\Sigma^{1/2}$ and translating by the random vector $\mu$.
\end{example}

\begin{example}[Conditional zonoids]
  If $\gae=\ve_\alpha$ is the right-average quantile, then
  $\Gae(\{X\})$ is the zonoid-trimmed region of $X\in\Lp[1](\R^d)$,
  see \cite{kos:mos97} and \cite{mos02}. In this case, the support function of
  $\Gae(\{X\})$ is equal to $\ve_\alpha(\langle X,u\rangle)$ for all
  $u\in\Sphere$. Its conditional variant yields a random convex
  body, which may be understood as a conditional zonoid of a
  multivariate distribution.  An application of this construction to the
  random vector $(1,X)$ in dimension $(d+1)$ yields lift zonoids (see
  \cite{mos02}) and their conditional variants.
\end{example}

\begin{example}[Conditional expectation and its maximum extension]
  If $\gae$ is the maximum extension of the conditional expectation,
  then $\Gae(\{X\}|\ssalg)$ is the conditional expectation of the
  random polytope $P$ obtained as the convex hull of the finite set
  $\{X_1,\dots,X_m\}$ built by i.i.d.\ copies of $X$. This
  expectation taken with respect a filtration appears in
  \cite{arar:ma23} as an example of a set-valued martingale. 
\end{example}

Let $\gae=\ve^{p,a}$ be defined at \eqref{eq:norm-gauge}.
Assume that $\E(X|\ssalg)=0$ a.s. for $X\in\Lp[1](\R^d)$.
Then the corresponding conditional depth-trimmed region is the largest
convex set such that 
\begin{displaymath}
  h(\Gae(\{X\}|\ssalg),W)= \Big(\E\big[|\langle
  X,W\rangle|^p_+\big|\ssalg]\Big)^{1/p}
\end{displaymath}
for all $W\in\Lp[0](\Sphere;\ssalg)$. 

The depth-trimmed regions based on considering expectiles have been
studied by \cite{cascos21:_expec}. Our construction provides their
conditional variant, which also have been mentioned in the cited
paper in view of application to the regression setting.


\section{Translations of cones}
\label{sec:cones-transl-rand}

\subsection{The case of a deterministic cone}
\label{sec:case-determ-cone}

Let $\bX=X+\cone$, where $X\in\Lp(\R^d)$ and $\cone$ is a
deterministic convex cone in $\R^d$ which is different from the origin
and the whole space.  In this case, $h(\bX,w)=\langle X,w\rangle$ if
$w$ belongs to the polar cone $\cone^o$ and otherwise
$h(\bX,w)=\infty$ a.s. Thus, $\BX=\cone^o$. 

In the unconditional setting, $\Gae(X+\cone)$ is the largest closed convex
set satisfying
\begin{displaymath}
  h(\Gae(X+\cone),w)\leq \gae(\langle X,w\rangle),\quad w\in \cone^o. 
\end{displaymath}
Then 
\begin{displaymath}
  \Gae(X+\cone)=\bigcap_{w\in\Sphere} H_w(\gae(X+\cone,w)))
  =\bigcap_{w\in \cone^o} H_w(\gae(\langle X,w\rangle)).
\end{displaymath}
If $\gae=\quantile^-$ is the quantile, we recover the construction
from \cite{ham:kos18}, namely, $\quantile^-(X+\cone)$ is the
lower $\cone$-quantile function $Q^{-}_{X,\cone}(\alpha)$ of $X$
introduced in \cite[Definition~4]{ham:kos18}. If $\gae$ is the
expectile, we recover the downward cone expectile from Definition~2.1
of \cite{linh:hamel23}. If $\gae$ is the dual to the expectile, we
recover the upward cone expectile from the same reference. Our
construction works for a number of further gauge functions and also in
the conditional setting. 

The following result immediately follows from (G5).

\begin{proposition}
  \label{prop:plus-cone}
  Let $\gae$ be a sublinear expectation defined on $\Lpb[p](\R)$. If
  $X\in\Lp(\R^d)$ with $\Gae(\{X\}|\ssalg)$ almost surely nonempty,
  and $\cone$ is a deterministic cone, then
  \begin{displaymath}
    \Gae(X+\cone|\ssalg)=\Gae(\{X\}|\ssalg)+\cone. 
  \end{displaymath}
\end{proposition}

With this result, the examples from
Section~\ref{sec:random-singletons} can be easily reformulated for the
case of random translations of deterministic cones. 

\begin{example}
  Let $\gae$ be a sublinear expectation which is sensitive with
  respect to infinity, e.g., the right-average quantile or expectile
  with $\tau\geq1/2$. Assume that $\cone^o$ is a subset of
  $\R_+^d$. Then
  \begin{displaymath}
    \gae(h(X+\cone,w))=\gae(\langle X,w\rangle)
    \leq \langle \gae(X),w\rangle=h(\gae(X)+\cone,w),
    \quad w\in \cone^o.
  \end{displaymath}
  Thus, $\Gae(X+\cone)=\gae(X)+\cone$. In particular, if
  $\gae=\ve_\alpha$ is the right-average quantile, then $\Gae(X+\cone)$ is
  the sum of the zonoid-trimmed region of $X$ and the cone $\cone$.
\end{example}

\subsection{Random cones translated by a deterministic point}
\label{sec:random-cones-point}

Let $\bX=x+\bC$, where $x$ is deterministic and $\bC$ is a random
cone. Since the gauge is translation equivariant, 
$\Gae(\bX|\ssalg) = x+\Gae(\bC|\ssalg)$.  So we assume without loss of
generality that $x=0$ and let $\bX=\bC$. The constant preserving
property (G2) implies that $\Gae(\bC|\ssalg)=\bC$ if $\bC$ is
$\ssalg$-measurable.

Since $\BX=\bC^o$, we have $h(\bC,W)=0$ if $W$ is a selection of
$\bC^o$. Otherwise, $h(\bC,W)=\infty$ on the event
$\{W\notin \bC^o\}$.

\begin{proposition}
  \label{prop:cone-random}
  Assume that the chosen gauge satisfies (g9). Then, for the corresponding
  set-valued conditional gauge, we have 
  \begin{equation}
    \label{eq:cone-core}
    \Gae(\bC|\ssalg)=(\core(\bC^o|\ssalg))^o. 
  \end{equation}
\end{proposition}
\begin{proof}
  Due to (g9), for any $W\in \Lp[\infty](\R^d;\ssalg)$, on the event
  $\big\{\Prob{W\notin \bC^o|\ssalg}>0\big\}$ the conditional gauge is
  $\gae(h(\bC,W)|\ssalg)=\infty$.  Therefore, $\Gae(\bC|\ssalg)$
  becomes the intersection of $H_{W}(0)$ over all
  $W\in\Lp[\infty](\R^d;\ssalg)$ such that $\Prob{W\in\bC^o|\ssalg}=1$
  almost surely.  In particular, it is possible to take $W$ which
  almost surely belongs to $\bC^o$. The family of such $W$ is the
  family of selections of the conditional core $\core(\bC^o|\ssalg)$,
  so that the intersection becomes the right-hand side of
  \eqref{eq:cone-core}. It remains to notice that for each
  $W\in\Lp[\infty](\R^d;\ssalg)$, the support function of
  $(\core(\bC^o|\ssalg))^o$ in direction $W$ is at most
  $\gae(h(\bC,W)|\ssalg)$. 
\end{proof}


In the unconditional setting and assuming that the chosen gauge
function satisfies (g9), $\Gae(\bC)$ is the polar cone to the set of
fixed points of $\bC^o$.  
By \citet[Proposition~5.5]{lep:mol19}, $(\core(\bC^o|\ssalg))^o$ equals
the conditional convex hull $\chull(\bX|\ssalg)$ of $\bC$. 
Since the essential supremum also satisfies (g9), we obtain as a
corollary of Proposition~\ref{prop:cone-random} that
$\esssup(\bC|\ssalg)$ equals the conditional convex hull of $\bC$,
which also follows from \eqref{eq:chull-cond}.

If $\gae$ is the essential infimum, then the situation is
different. Then, $\essinf(h(\bC,W)|\ssalg)=\infty$ on the
event $\big\{\Prob{W\notin\bC^o|\ssalg}=1\big\}$, meaning that
$\Prob{h(\bC,W)=\infty|\ssalg}=1$ a.s.\ and implying that
$W \notin\bC^o$ almost surely. This is the polar set to
$\esssup(\bC^o|\ssalg)$ and so
\begin{displaymath}
  \essinf(\bC|\ssalg)=(\esssup(\bC^0|\ssalg))^o=\core(\bC|\ssalg). 
\end{displaymath}
In particular, $\essinf(\bC)$ is the set of fixed points of $\bC$.


From our list of examples of gauge functions, the only ones not
satisfying (g9) are the quantiles and left-average quantiles. So first
suppose that $\gae=\quantile^-$. Consider the random variable
$\zeta_t$, where $\zeta_t=\infty$ with probability $1-t$ and otherwise
$\zeta_t=0$ for $t\in(0,1)$. Define $r(t) =
\quantile^-(\zeta_t)$. Then $r(t)=0$ if $t\ge \alpha$ and oherwise
$r(t)=\infty$. Thus, in the unconditional setting for $w\in\R^d$,
\begin{displaymath}
  \quantile^-(h(\bC,w))=r\big(\Prob{w\in\bC^o}\big)=
  \begin{cases}
    0, & \text{if}\; \Prob{w\in\bC^o}\ge\alpha,\\
    \infty, & \text{otherwise}. 
  \end{cases}
\end{displaymath}
This set-valued gauge is the polar to the cone of all $w\in \R^d$ such
that $\Prob{w\in\bC^o}\ge\alpha$.

The set
\begin{equation}
  \label{eq:Qalpha}
  Q_\alpha(\bC^o)=\big\{w: \Prob{w\in\bC^o}\geq \alpha\big\}
\end{equation}
is called the Vorob'ev quantile of $\bC^o$ at level $\alpha$, see
\citet[Section~2.2.2]{mo1}. Note that $Q_\alpha(\bC^o)$ is itself a
cone and $Q_\alpha(\bC^o)$ converges to the set of all $x$ such that
$\Prob{x\in\bC^o}>0$ as $\alpha\downarrow0$. Thus,
\begin{displaymath}
  \quantile^-(\bC)=\big(Q_\alpha(\bC^o)\big)^o,
\end{displaymath}
that is, the gauge of $\bC$ is the polar set to
$Q_\alpha(\bC^o)$. The same expression holds for left-average
quantiles, since $\ue_\alpha(\zeta_t)=\quantile^-(\zeta_t)$.


\subsection{Random translations of random cones}
\label{sec:rand-transl-rand}

Let $\bX=X+\bC$ for a $p$-integrable random vector $X$ and a (possibly
random) cone $\bC$. There are two natural ways to introduce the
conditional $\sigma$-algebra $\ssalg$ in this setting. If $\ssalg$ is
generated by $X$ (more generally, if $X$ is $\ssalg$-measurable), then
(G5) yields that
\begin{displaymath}
  \Gae(X+\bC|\ssalg)=X+\Gae(\bC|\ssalg). 
\end{displaymath}
Then the arguments from Section~\ref{sec:random-cones-point} apply.
If $\bC$ is $\ssalg$-measurable, then (G5) yields that
\begin{displaymath}
  \Gae(X+\bC|\ssalg)=\Gae(X|\ssalg)+\bC.
\end{displaymath}

In the remainder of this section we adapt the unconditional setting.
If (g9) holds, then 
\begin{displaymath}
  \Gae(X+\bC)=\bigcap_{w\in\Sphere}H_w(\gae(h(X+\bC,w))
  =\bigcap_{w\in\cone}H_w(\gae(\langle X,w\rangle)),
\end{displaymath}
where $\cone=\core(\bC^o)$ is the set of points $w$ which almost
surely belong to $\bC^o$.  Without (g9), the same holds with $\cone$
replaced by $Q_\alpha(\bC^o)$ from \eqref{eq:Qalpha}.

\begin{example}
  \label{ex:bid}
  Consider the convex cone $\bC$ in $\R^2$ with points $(-\kappa_1,1)$
  and $(1,-\kappa_2)$ on its boundary, where $\kappa_1$, $\kappa_2$
  are positive random variables such that $\kappa_1\kappa_2\geq1$ a.s.
  This cone describes the family of portfolios on two assets available
  at price zero and such that $\kappa_1$ units of the first asset can
  be exchanged for one unit of the second asset and also $\kappa_2$
  units of the second asset can be exchanged for a single unit of the
  first asset. The higher the value of $\kappa_1\kappa_2$ is, the more
  transaction costs are paid when exchanging the assets. If the gauge
  function satisfies (g9), then $\Gae(\bC)=(\core(\bC^o))^o$. If
  $\essinf\kappa_i=k_i$, $i=1,2$, and $k_1k_2\geq1$, then
  $\core(\bC^o)$ is the cone with $(k_2,1)$ and $(1,k_1)$ on its
  boundary, so that $\Gae(\bC)$ is the cone with points $(1,-k_2)$ and
  $(-k_1,1)$ on its boundary. If $k_1k_2<1$, then $\core(\bC^o)=\R_+^2$ and so
  $\Gae(\bC)=\{0\}$.  If the essential suprema of $\kappa_1$ and
  $\kappa_2$ are infinite, then $\essinf(\bC)=\R_-^2$.  If $\gae$ is a
  quantile, then $\Gae(\bC)$ is the polar to the Vorob'ev quantile of
  $\bC^o$, which is given by the cone with points
  $(1,\quantile^-(\kappa_1))$ and $(\quantile^-(\kappa_2),1)$ on its
  boundary.

  Assume that $\bX=X+\bC$ for $X\in\Lpb(\R^d)$. Then
  $h(\bX,w)=\langle X,w\rangle$ if $w\in\bC^o$ and otherwise
  $h(\bX,w)=\infty$. If the gauge function satisfies (g9), then
  $\Gae(\bX)$ is the sum of $\Gae(\{X\}$ (which is the corresponding
  depth-trimmed region of $X$) and $\Gae(\bC)$ as described above. If
  the gauge is a quantile $\quantile^-$ we have
  \begin{displaymath}
    \quantile^-(h(\bX,w))
    =\inf\big\{t\in\R: \Prob{\langle X,w\rangle \leq t,w\in\bC^o}\geq
    \alpha\big \}. 
  \end{displaymath}
\end{example}

\begin{example}
  Let $X$ be normally distributed in $\R^d$ with mean $\mu=(1,2)$, the
  standard deviations $\sigma_1=0.3$, $\sigma_2=0.5$ and the
  correlation coefficient $\rho=0.6$. Denote by $\Sigma$ the
  covariance matrix of $X$ and note that 
  \begin{displaymath}
    \Sigma= \begin{pmatrix}
    0.09&0.09\\0.09&0.25
    \end{pmatrix}\,.
  \end{displaymath}
  Let $\bC$ be a random cone that
  contains $\R_+^2$ and that has points $(-1,\pi)$ and $(\pi,-1)$ on its
  boundary, where $(\pi-2)$ is lognormally
  distributed with mean zero and the volatility
  $\sigma=0.2$. Note that
  $\bC^o$ is the cone with points
  $(-\pi,-1)$ and $(-1,-\pi)$ on its boundary, so that
  $\core(\bC)^o=\cone$ is the cone with points $(-2,-1)$ and $(-1,-2)$ on
  its boundary. 
  Below we calculate $\Gae(X+\bC)$ for different choices
  of the underlying gauge: the expectation, the quantile and a lower 
  average quantile with $\alpha=0.1$, the upper average quantiles with
  $\alpha=0.9$, and the norm-based one with $p=2$ and $a=1$.  If $u\in \cone$ and $\zeta$ is a standard normal random variable, then
  \begin{align*}
    \E h(X+\bC,u)
    &=\langle \mu,u\rangle,\\
    \ve_{0.9}(h(X+\bC,u))
    &=\langle\mu,u\rangle+\langle \Sigma u,u\rangle^{1/2}
      \ve_{0.9}(\zeta)
      =\langle\mu,u\rangle+\langle \Sigma u,u\rangle^{1/2}
      (1-\alpha)^{-1}\varphi(\Phi^{-1}(\alpha))\\
    &\approx\langle\mu,u\rangle+1.75 \langle \Sigma u,u\rangle^{1/2},\\
    \ue_{0,1}(h(X+\bX,u))
    &=\langle\mu,u\rangle+\langle \Sigma u,u\rangle^{1/2}
       \ue_{0,1}(\zeta) = \langle\mu,u\rangle-
      \langle \Sigma u,u\rangle^{1/2}\alpha^{-1}\varphi(\Phi^{-1}(\alpha)) \\
      &\approx \langle\mu,u\rangle-1.75 \langle \Sigma u,u\rangle^{1/2},\\
    \ve^{2,1}(h(X+\bC,u))
    &=\langle\mu,u\rangle+(\E \langle
      X,u\rangle_+^2)^{1/2}=\langle\mu,u\rangle
      +\langle \Sigma u,u\rangle^{1/2}/\sqrt{2}. 
  \end{align*}
  All these gauges are infinite for all $u\notin\cone$. Thus,
  $\E(X+\bC)=\mu+\cone^o$. 

  The Vorob'ev quantile $Q_\alpha(\bC^o)$ at level $\alpha=0.1$ of
  $\bC^o$ is the cone which is a subset of $\R_-^2$ and has points
  $(-2.77,-1)$ and $(-1,-2.77)$ on its boundary. If
  $u\in Q_\alpha(\bC^o)$, then
  \begin{align*}
   \mathsf{q}^-_{0.1}(h(X+\bX,u))
    &= \langle\mu,u\rangle+\langle \Sigma u,u\rangle^{1/2}
      \Phi^{-1}(\alpha) \approx \langle\mu,u\rangle-1.28 \langle \Sigma u,u\rangle^{1/2}.
  \end{align*}
\end{example}

\section*{Acknowledgements}

The authors are grateful to Ignacio Cascos for comments on earlier
versions of this paper.

\section*{Compliance with ethical standards}

No funds, grants, or other support was received.
The authors have no relevant financial or non-financial interests to disclose.
The authors declare that they have no Conflict of interest.


\end{document}